\documentclass[letterpaper,10pt,reqno,onefignum,onetabnum]{amsart}
\usepackage[english]{babel}
\usepackage{amsmath}
\usepackage{amsthm}
\usepackage{verbatim}
\usepackage{mathrsfs}
\usepackage{bm} 
\usepackage[dvipsnames]{xcolor}
\usepackage{hyperref}
\hypersetup{
     colorlinks = true,
     linkcolor = OliveGreen,
     anchorcolor = OliveGreen,
     citecolor = OliveGreen,
     filecolor = OliveGreen,
     urlcolor = OliveGreen
     }
\usepackage{algorithm}
\usepackage{algorithmic}
\usepackage{float}
\usepackage{lipsum}
\usepackage{amsfonts}
\usepackage{amssymb}
\usepackage{graphicx}
\usepackage{epstopdf}
\usepackage{cases}
\usepackage{multirow}
\ifpdf
  \DeclareGraphicsExtensions{.eps,.pdf,.png,.jpg}
\else
  \DeclareGraphicsExtensions{.eps}
\fi
\usepackage{verbatim}
\usepackage{mathrsfs}
\usepackage{bm}

\newtheorem{teo}{Theorem}[section]
\newtheorem{prop}[teo]{Proposition}

\newtheorem{cor}[teo]{Corollary}
\newtheorem{pro}[teo]{Problem}
\newtheorem{algo}[teo]{Algorithm}

\newtheorem{rem}[teo]{Remark}
\newtheorem{ejem}[teo]{Example}

\usepackage{lipsum}
\usepackage{amsfonts}
\usepackage{amssymb}
\usepackage{graphicx}
\usepackage{epstopdf}
            
\usepackage{cases}
\usepackage{multirow}
\usepackage{algorithmic}
\usepackage{tikz}
\usetikzlibrary{matrix}
\usetikzlibrary{arrows}
\ifpdf
  \DeclareGraphicsExtensions{.eps,.pdf,.png,.jpg}
\else
  \DeclareGraphicsExtensions{.eps}
\fi
\usepackage{verbatim}
\usepackage{mathrsfs}
\usepackage{bm} 
\usepackage{color}
\usepackage{caption}
\usepackage{subcaption}

\newcommand{\N}{\mathbb N}

\newcommand{\R}{\mathbb R}
\newcommand{\J}{\mathcal H}
\renewcommand{\H}{\mathcal{H}}
\newcommand{\G}{\mathcal G}
\newcommand{\K}{\mathcal K}

\newcommand{\HH}{{\bm{\mathcal{H}}}}
\newcommand{\vv}{{\bm V}}

\newcommand{\TT}{{\bm T}}

\newcommand{\id}{\textnormal{Id}}

\newcommand{\s}{\sigma}
\newcommand{\T}{\tau}
\newcommand{\weak}{\rightharpoonup}
\newcommand{\ran}{\textnormal{ran}\,}
\newcommand{\dom}{\textnormal{dom}\,}

\newcommand{\zer}{\textnormal{zer}}
\newcommand{\fix}{\textnormal{Fix}\,}
\newcommand{\gra}{\textnormal{gra}\,}

\newcommand{\argm}[1]{\underset{#1}{\argmin\, }}

\newcommand{\scal}[2]{{\left\langle{{#1}\mid{#2}}\right\rangle}}

\newcommand{\menge}[2]{\big\{{#1}~\big |~{#2}\big\}} 
\newcommand{\pinf}{\ensuremath{{+\infty}}}

\newcommand{\RPP}{\ensuremath{\left]0,+\infty\right[}}

\newcommand{\RX}{\ensuremath{\left]-\infty,+\infty\right]}}

\newcommand{\sri}{\ensuremath{\text{\rm sri}\,}}
\newcommand{\prox}{\ensuremath{\text{\rm prox}}}
\newcommand{\weakly}{\ensuremath{\:\rightharpoonup\:}}
\usepackage{geometry}
\numberwithin{equation}{section}

\DeclareFontEncoding{FMS}{}{}
\DeclareFontSubstitution{FMS}{futm}{m}{n}
\DeclareFontEncoding{FMX}{}{}
\DeclareFontSubstitution{FMX}{futm}{m}{n}
\DeclareSymbolFont{fouriersymbols}{FMS}{futm}{m}{n}
\DeclareSymbolFont{fourierlargesymbols}{FMX}{futm}{m}{n}
\DeclareMathDelimiter{\nr}{\mathord}{fouriersymbols}{152}{fourierlargesymbols}{147}

\DeclareMathOperator*{\argmin}{arg\,min}
\DeclareMathDelimiter{\nr}{\mathord}{fouriersymbols}{152}{fourierlargesymbols}{147}

\newcommand{\lb}[1]{{\color{black}{#1}}}

\title[{Split-Douglas-Rachford and Split-ADMM}]{Split-Douglas-Rachford \lb{algorithm} for composite 
  monotone 
  inclusions and Split-ADMM}

\author{Luis M. Brice\~no-Arias \& Fernando Rold\'an}
\address{Departamento de Matem\'{a}tica, Universidad T\'{e}cnica Federico Santa Mar\'{i}a, Avenida Espa\~{n}a 1680, Valpara\'{i}so, Chile}
\email{luis.briceno@usm.cl, fernando.roldan@usm.cl}


\begin{document}

\begin{abstract}In this paper we provide a generalization of the
Douglas-Rachford splitting (DRS) \lb{and the primal-dual algorithm 
\cite{condat,vu}} for solving monotone inclusions
in a real Hilbert space involving a general linear operator. The 
proposed method \lb{allows for primal and dual non-standard metrics 
and }activates the linear operator separately from 
the monotone operators appearing in the inclusion\lb{. I}n the 
simplest case when the linear operator \lb{has full range}, it reduces 
to classical DRS. Moreover, the weak convergence of primal-dual 
sequences to \lb{a Kuhn-Tucker point} 
is guaranteed, generalizing the main result in \cite{svaiter}. 
Inspired \lb{by} \cite{gabay83}, we \lb{also}
derive a new Split-ADMM (SADMM) by applying our method to the 
dual of a  convex optimization problem involving a linear operator  
which can be expressed as the composition of two linear
 operators. The proposed SADMM activates one linear operator 
 implicitly and the other \lb{one} explicitly\lb{,} and we recover ADMM 
 when the 
 latter is set as the identity. Connections and comparisons of our 
 theoretical results with respect to the literature are provided for 
 the main algorithm and SADMM\lb{. T}he flexibility and efficiency 
 of \lb{both methods} is illustrated via a 
numerical simulations \lb{in total variation image restoration and} a 
sparse 
minimization problem.
\par
\bigskip

\end{abstract}

\maketitle

\section{Introduction}
In this paper we focus on a splitting algorithm for solving the 
following 
primal-dual monotone inclusion.
\begin{pro}\label{prob2}
	Let $\H$ and $\G$ be real Hilbert spaces, let 
	$A\colon \H\rightarrow 2^{\H}$ and 
	$B \colon \G \rightarrow 2^{\G}$  be  maximally monotone 
	operators, and let $L\colon \H \rightarrow \G$ be  a non-zero
	linear 
	bounded operator. The problem is to find $(\hat{x},\hat{u}) \in 
	\bm{Z}$, where
	\begin{equation}\label{Z}
		\bm{Z}=\menge{(\hat{x},\hat{u}) \in \H \times \G}{0 \in A\hat{x}+ L^*\hat{u},\:
			0 \in B^{-1}\hat{u}-L\hat{x}}
	\end{equation}
	is assumed to be non-empty.
\end{pro}
This problem arises naturally in several problems in partial 
differential equations coming from mechanical problems 
\cite{gabay83,GlowinskyMorrocco75,Goldstein64}, differential 
inclusions \cite{Aubin1990,Showalter}, game theory \cite{Nash}, 
among 
other disciplines. The set $\bm{Z}$ is 
the collection of Kuhn-Tucker points \cite[Problem~26.30]{1},
which is also known as \textit{extended 
	solution set} (see, e.g., \cite{Vu15} and \cite{EckSvaiter,svaiter} for 
the case when $L=\id$). 

It follows from \cite[Proposition~2.8]{skew} that 
any solution $(\hat{x},\hat{u})$ to Problem~\ref{prob2} satisfies 
that $\hat{x}$ is a solution to the primal inclusion
\begin{equation}
	\label{e:priminc}
	\text{find}\quad x\in\H\quad\text{such that}\quad 0\in Ax+L^*BLx
\end{equation}
and $\hat{u}$ is solution to the dual inclusion
\begin{equation}
	\label{e:dualinc}
	\text{find}\quad u\in\G\quad\text{such that}\quad 0\in 
	B^{-1}u-LA^{-1}(-L^*u).
\end{equation}
Conversely, if $\hat{x}$ is a solution to \eqref{e:priminc} then
there exists $\tilde{u}$ solution to \eqref{e:dualinc} such that
$(\hat{x},\tilde{u})\in\bm{Z}$ and the dual argument also holds.
In the particular case when $A=\partial f$ and 
$B=\partial g^*$, for proper convex lower semicontinuous 
functions $f\colon\H\to\RX$ and $g\colon\G\to\RX$, 
any solution $\hat{x}$ to \eqref{e:priminc} is a solution to the primal
convex optimization problem
\begin{equation}
	\label{e:primopti}
	\min_{x \in \H}\big(f(x)+g(Lx)\big),
\end{equation}
any solution $\hat{u}$ to 
\eqref{e:dualinc} is a  solution to the dual problem
\begin{equation}
	\label{e:dualopti}
	\min_{u \in \G}\big(g^*(u)+f^*(-L^*u)\big),
\end{equation}
and the converse holds under standard qualification conditions
(see, e.g., \cite{skew}).
Problems \eqref{e:primopti} and \eqref{e:dualopti} model several 
image processing 
problems as image restoration and denoising 
\cite{ChambLions97,Pustelnik2019,Daube04,Liu2014,Pang2017,Sha2020},
traffic theory 
\cite{Nets1,Fuku96,GafniBert84}, among others.

In the case when $L=\id$, Problem~\ref{prob2} is solved by the 
Douglas-Rachford 
splitting (DRS) \cite{lions-mercier79}, which is a classical 
algorithm inspired 
from a numerical method for solving linear systems appearing in 
discretizations of PDEs \cite{DR1956}. Given $z_0\in\H$ and 
$\tau>0$, DRS generates the sequence 
$(z_n)_{n\in\N}\subset \H$ 
via the recurrence
\begin{equation}
	\label{e:DRSintro}
	(\forall n\in\N)\quad z_{n+1}=J_{\T B } (2J_{\T A}z_n - z_n 
	)+z_n- J_{\T A}z_n,
\end{equation}
and $z_n\weakly \hat{z}$ for some $\hat{z}\in\H$ 
such that $J_{\tau A}\hat{z}$ is a zero of $A+B$ 
\cite[Theorem~1]{lions-mercier79}, where we 
denote the resolvent of  $M\colon\H\to 2^{\H}$ by 
$J_M=(\id+M)^{-1}$. Under additional assumptions, such as
weak lower semicontinuity of $J_{\tau A}$ or maximal 
monotonicity 
of $A+B$,
the weak convergence of 
the \textit{shadow sequence} $(J_{\tau 
	A}{z_n})_{n\in\N}$ to a zero of $A+B$ is guaranteed in 
\cite[Theorem~1]{lions-mercier79}. 
More than thirty years later, the weak convergence of the 
shadow sequence to a solution is proved in \cite{svaiter} without 
any further assumption.

In the general case when $L\ne\id$, a drawback of DRS is that 
the maximal monotonicity of $L^* B L$ is needed in order to 
ensure the
weak convergence of $(z_n)_{n\in\N}$ and the computation of 
its resolvent at each iteration usually leads to sub-iterations, at 
exception of very particular cases. 
Several algorithms in the literature including 
\cite{bot2018,BOT2,bot2013,skew,Siopt2,vu} 
split the influence 
of the linear operator $L$ from the monotone operators, avoiding 
sub-iterations. In particular, we highlight the primal-dual splitting 
(PDS) proposed in \cite{vu}, which generates a sequence in 
$\H\times\G$ via the recurrence
\begin{equation}
	\label{e:condatvu}
	(\forall n\in\N)\quad 
	\begin{array}{l}
		\left\lfloor
		\begin{array}{l}
			{x}_{n+1}= J_{\T A} ( {x}_n -\T L^* {v}_n)\\
			{v}_{n+1}=J_{\s B^{-1}}( {v}_n+\s L(2{x}_{n+1}-{x}_n)),
		\end{array}
		\right.
	\end{array}
\end{equation}
for some initial point $(x_0,v_0)\in\H\times\G$ and strictly positive 
step-sizes satisfying $\tau\sigma\|L\|^2<1$. 

In the context of convex optimization, it is well known that 
DRS applied to \eqref{e:dualopti} leads to the alternating direction 
method of multipliers (ADMM) 
\cite{gabay83,gabaymercier,GlowinskyMorrocco75}, whose first step 
needs 
sub-iterations in general. This drawback is overcome by the 
splitting methods proposed in 
\cite{bot2018,BOT2,bot2013,Siopt2,cp,yuan,Molinari2019}. 
In particular, the algorithm
proposed in \cite{cp} coincides with PDS in
\eqref{e:condatvu} in the optimization setting and its 
convergence is guaranteed if  $\tau\sigma\|L\|^2<1$. 
In \cite{condat}, the convergence of the sequences generated by 
\eqref{e:condatvu}
with step-sizes satisfying the limit condition $\tau\sigma\|L\|^2=1$ 
is studied in finite 
dimensions. This limit case is important because the algorithm 
improves its efficiency as the parameters approach the boundary 
(see Section~\ref{sec:TV}), it has the advantage of tuning only 
one parameter, 
and the algorithm reduces to DRS 
and ADMM when $L=\id$ and $\tau\sigma=1$ \cite[Section~4.2]{cp}. 
Furthermore, a preconditioned version 
of \eqref{e:condatvu} in the optimization context is proposed in 
\cite{pockcham}. In this extension, 
$\tau \id$ and $\s\id$ are 
generalized to strongly monotone self-adjoint linear operators 
$\varUpsilon$ and $\Sigma$, respectively, and the convergence is 
guaranteed  under the 
condition $\|\Sigma^{\frac{1}{2}}LT^{\frac{1}{2}}\|<1$. A preconditioned 
version of \eqref{e:condatvu} for monotone inclusions is 
derived in \cite{CombVu14}.

In this paper we propose and study the following splitting 
algorithm for solving Problem~\ref{prob2}, which is a 
generalization of DRS when $L\ne \id$ and of 
\cite{CombVu14,vu}.
\begin{algo}[Split-Douglas-Rachford (SDR)]\label{alg:SDR}
	In the context of Problem~\ref{prob2}, let $(x_0,u_0) \in \H\times\G$, 
	let $\Sigma\colon \G\to\G$ and $\varUpsilon\colon\H\to\H$ be 
	strongly 
	monotone self-adjoint linear operators such that 
	$U=\varUpsilon^{-1}-L^*\Sigma L$
	is monotone. Consider the recurrence:
	\begin{equation}\label{eq:algoSDR} 
		(\forall n \in \N)\quad
		\begin{array}{l}
			\left\lfloor
			\begin{array}{l}
				v_{n}=\Sigma(\id-J_{\Sigma^{-1}B})
				(Lx_n+\Sigma^{-1}u_n)\\
				z_{n}=x_{n}-\varUpsilon L^*v_{n} \\
				x_{n+1}=J_{\varUpsilon A} {z_n}\\
				u_{n+1}=  \Sigma L(x_{n+1}-x_{n})+v_{n}.
			\end{array}
			\right.
		\end{array}
	\end{equation}
\end{algo}
Note that
Algorithm~\ref{alg:SDR} splits the influence of the linear operator 
from the monotone operators and, by storing $(Lx_n)_{n\in\N}$,
only one activation of $L$ is needed at each iteration.
Moreover, 
in the case when $\ran L =\G$, 
we prove in Proposition~\ref{prop:GDRtoDR} that 
\eqref{eq:algoSDR} reduces to a preconditioned version of DRS in
\eqref{e:DRSintro}, in which case $J_{\varUpsilon L^*BL}$ has a 
closed formula depending on the resolvent of $B$. Other 
preconditioned versions of 
DRS are used for solving structured convex optimization problems in 
\cite{bot2013,BrediesSun15,BrediesSun15TV,Yang21}, but they do 
not reduce 
to DRS when $L=\id$.
Without any further assumptions than those in Problem~\ref{prob2}, 
we guarantee the weak convergence of the sequence 
$\big((x_n,u_n)\big)_{n\in\N}$
generated by Algorithm~\ref{alg:SDR} 
to a point in $\bm{Z}$, generalizing 
the result in \cite{svaiter} to the case when $L\ne\id$. 
In the particular case when 
$\|\Sigma^{\frac{1}{2}}LT^{\frac{1}{2}}\|<1$, we obtain 
a reduction of Algorithm~\ref{alg:SDR} to the preconditioned PDS in 
\cite{pockcham} and, when 
$\|\Sigma^{\frac{1}{2}}LT^{\frac{1}{2}}\|=1$, we generalize 
\cite[Theorem~3.3]{condat} to
monotone inclusions and infinite 
dimensions considering non-standard metrics. We also provide a 
numerical comparison of 
Algorithm~\ref{alg:SDR} with several methods available in the 
literature in a total variation image reconstruction problem.

Another contribution of this manuscript is a generalization of 
ADMM in the convex optimization context, by applying 
Algorithm~\ref{alg:SDR} to the dual problem of 
\eqref{e:primopti} when $L=KT$,
for some non-trivial linear operators $T$ and 
$K$. This splitting, called Split-ADMM (SADMM), 
allows us to solve \eqref{e:primopti} by activating $T$ implicitly 
and 
$K$ explicitly. SADMM reduces to 
the classical ADMM in the case when $K=\id$, 
$\Sigma=\sigma\id$, and $\varUpsilon=\tau\id$ and, in the case 
when $T=\id$, it is a fully explicit algorithm which splits
the influence of the linear operator in the first step of ADMM. We 
prove the weak convergence of SADMM, generalizing results in
\cite{ecksteinthesis,gabay83,gabaymercier}. We also prove the 
equivalence 
between SDR and SADMM, generalizing some results in 
\cite{russell2016,ecksteinthesis,gabay83,gabaymercier,moursi2019}
to the case when $L\ne\id$. 
In addition, we provide a version of SADMM able to deal with two 
linear operators as in \cite{BrediesSun17}. The resulting method is a 
non-standard metric version of several ADMM-type algorithms in 
\cite{bot2018,BrediesSun17,shefi,Osher11} and it can be 
seen
as an augmented Lagrangian method with a non-standard metric.
We also illustrate the efficiency of SADMM by 
comparing its numerical performance in an academical sparse 
minimization 
example in which the matrix $L$ be factorized as $L=KT$ 
from its singular value 
decomposition (SVD). We show 
that the computational time may be drastically reduced by using 
SADMM with a suitable factorization of $L$.

The paper is organized as follows. In Section~\ref{sec:notation} 
we set our notation. 
In Section~\ref{sec:primaldual} we provide the proof of 
convergence of SDR and we connect our results
with the literature. 
In Section~\ref{sec:SADMM} we derive the SADMM,
we provide several theoretical results, and we compare them with 
the literature in convex optimization. Finally,
in Section~\ref{sec:numer} we provide numerical simulations 
illustrating the efficiency of SDR and SADMM.

\section{Notations and Preliminaries}
\label{sec:notation}
Throughout this paper $\H$ and $\G$ are real Hilbert spaces with 
the scalar 
product $\scal{\cdot}{\cdot}$ and associated norm $\|\cdot 
\|$. The identity operator on $\H$ is denoted by $\id$. Given a 
linear bounded operator $L:\H \to \G$, we denote its adjoint 
by $L^*\colon\G\to\H$, its kernel by $\ker L$, and its range by 
$\ran L$.
The symbols $\weakly$ and $\to$ denote the weak and strong 
convergence, respectively.
Let $D\subset \H$ be non-empty and let $T: D \rightarrow \H$. 
The set of fixed points of $T$ is $\fix T = \menge{x \in D}{x=Tx}$.
Let $\beta \in \left]0,+\infty\right[$. The operator $T$ is  
$\beta-$strongly monotone if, for every $x$ and $y$ in $D$, we have
$\scal{x-y}{Tx-Ty}\ge \beta\|x-y\|^2$,
it is nonexpansive if, for every $x$ and $y$ in $D$, we have
$\|Tx - Ty \|\leq \|x-y\|$,
it is firmly nonexpansive  if 
\begin{equation}\label{def:quasifirmnonexpansive}
	(\forall x \in D) (\forall y \in D)\ \ \|Tx - Ty \|^2 \leq 
	\|x-y\|^2-\|(\id-T)x-(\id-T)y\|^2,
\end{equation}
and it is firmly quasinonexpansive if, for every $x\in D$ and
$y\in\fix T$, we have $\|Tx - y \|^2 \leq 
\|x-y\|^2-\|Tx-x\|^2$.
Let $A:\H \rightarrow 2^{\H}$ be a set-valued operator. 
The inverse of $A$ is 
$A^{-1} \colon u \mapsto \menge{x \in \H}{u \in Ax}$.
The domain, range, graph, and zeros of $A$ 
are 
$\dom\, A = \menge{x \in \H}{Ax \neq  \varnothing}$,
$\ran\, A = \menge{u \in 
	\H}{(\exists x \in \H)\,\, u \in Ax}$, 
$\gra A = \menge{(x,u) \in \H \times \H}{u \in Ax}$,
and $\zer A = 
\menge{x \in \H}{0 \in Ax}$, respectively. 
The operator $A$  is monotone if, for every $(x,u)$ and $(y,v)$
in $\gra A$, we have $\scal{x-y}{u-v} 
\geq 0$
and $A$ is maximally monotone if it is monotone and its graph is 
maximal in the sense of 
inclusions among the graphs of monotone operators.
The resolvent of a maximally monotone operator $A$ is 
$J_A=(\id+A)^{-1}$, which is firmly nonexpansive and satisfies 
$\fix J_A=\zer A$.

For every self-adjoint monotone linear operator $U\colon\H\to\H$, 
we define $\|\cdot\|_U=\sqrt{\scal{\cdot}{\cdot}_U}$, where 
$\scal{\cdot}{\cdot}_U \colon (x,y) \to 
\scal{x}{U y}$ is bilinear,  positive semi-definite, 
symme\-tric. For every $x$ and $y$ in $\H$, we have
\begin{equation}
	\label{e:scalU}
	\|x-y\|_U^2=\|x\|_U^2-2\scal{x}{y}_U+\|y\|^2_U.
\end{equation}

We denote by $\Gamma_0(\H)$ the class of proper lower 
semicontinuous convex functions $f\colon\H\to\RX$. Let 
$f\in\Gamma_0(\H)$.
The Fenchel conjugate of $f$ is 
defined by $f^*\colon u\mapsto \sup_{x\in\H}(\scal{x}{u}-f(x))$,
$f^*\in \Gamma_0(\H)$,
the subdifferential of $f$ is the maximally monotone operator
$\partial f\colon x\mapsto \menge{u\in\H}{(\forall y\in\H)\:\: 
	f(x)+\scal{y-x}{u}\le f(y)}$,
$(\partial f)^{-1}=\partial f^*$, 
and we have  that $\zer\,\partial f$ is the set of 
minimizers of $f$, which is denoted by $\arg\min_{x\in \H}f$. 
Given a strongly monotone self-adjoint linear operator 
$\varUpsilon\colon\H\to\H$, we denote by 
\begin{equation}
	\label{e:prox}
	\prox^{\varUpsilon}_{f}\colon 
	x\mapsto\argm{y\in\H}\big(f(y)+\frac{1}{2}\|x-y\|_{\varUpsilon}^2\big),
\end{equation}
and by $\prox_{f}=\prox^{\id}_{f}$. We have 
$\prox^{\varUpsilon}_f=J_{\varUpsilon^{-1}\partial f}$ 
\cite[Proposition~24.24(i)]{1}
and it is single valued since 
the objective function in 
\eqref{e:prox} is strongly convex. Moreover, it follows from 
\cite[Proposition~24.24]{1} that
\begin{equation}
	\label{e:Moreau_nonsme}
	\prox^{\varUpsilon}_{f}=
	\id-\varUpsilon^{-1}\,  
	\prox^{\varUpsilon^{-1}}_{f^*}\, 
	\varUpsilon=\varUpsilon^{-1}\,(\id-  
	\prox^{\varUpsilon^{-1}}_{f^*})\, 
	\varUpsilon.
\end{equation}
Given a non-empty closed convex set $C\subset\H$, we denote by 
$P_C$ the projection onto $C$, by
$\iota_C\in\Gamma_0(\H)$ the indicator function of $C$, which 
takes the value $0$ in $C$ and $\pinf$ otherwise, we denote
by $N_C=\partial \iota_C$ the normal cone to $C$, 
and by $\sri C$ its strong relative interior. 
For further properties of monotone operators,
nonexpansive mappings, and convex analysis, the 
reader is referred to \cite{1}. 

We finish this section with a result involving monotone linear 
operators, which is useful for the connection of our algorithm and 
\cite{pockcham}.
\begin{prop}
	\label{p:linearop}
	Let $\H$ and $\G$ be real Hilbert spaces, let 
	$\varUpsilon\colon\H\to\H$ 
	and 
	$\Sigma\colon\G\to\G$ be
	strongly monotone self-adjoint linear operators, and set
	\begin{equation}
		\vv\colon\H\oplus\G\to\H\oplus\G\colon (x,u)\mapsto 
		(\varUpsilon^{-1}x-L^*u,\Sigma^{-1}u-Lx).
	\end{equation} 
	Then, the following statements are equivalent.
	\begin{enumerate}
		\item\label{p:linearop1} 
		$\varUpsilon^{-1}-L^*\circ \Sigma\circ L$ is monotone.
		\item\label{p:linearop2} 
		$\|\Sigma^{\frac{1}{2}}\circ L\circ \varUpsilon^{\frac{1}{2}}\|\le 1$.
		\item\label{p:linearop3} 
		$\|\varUpsilon^{\frac{1}{2}}\circ L^*\circ \Sigma^{\frac{1}{2}}\|\le 1$.
		\item \label{p:linearop4}
		$\Sigma^{-1}-L\circ \varUpsilon\circ L^*$ is monotone.
		\item\label{p:linearop5}
		For every $(x,u)\in\H\times\G$,
		\begin{equation}
			\label{e:vcocomax}
			\scal{(x,u)}{\vv(x,u)}\ge\max\big\{\|\varUpsilon^{-1}u-L^*x\|_{\varUpsilon}^2,
			\|\Sigma^{-1}u-Lx\|_{\Sigma}^2\big\}.
		\end{equation}
	\end{enumerate}
	Moreover, if any of the statements above holds, $\vv$ is 
	$\frac{\tau\sigma}{\tau+\sigma}-$cocoercive, where $\tau>0$
	and $\sigma>0$ are the strong monotonicity constants of 
	$\varUpsilon$
	and $\Sigma$, respectively.
\end{prop}
\begin{proof}
	\ref{p:linearop1}$\Leftrightarrow$\ref{p:linearop2}:
	Since $\Sigma$ and $\varUpsilon$ are strongly monotone, linear, and 
	self-adjoint,
	it follows from \cite[Theorem~p.~265]{RieszNagy} that there exists 
	strongly monotone,
	linear, self-adjoint operators $\Sigma^{\frac{1}{2}}$ and 
	$\varUpsilon^{\frac{1}{2}}$ such that $\Sigma=\Sigma^{\frac{1}{2}}\circ 
	\Sigma^{\frac{1}{2}}$ and $\varUpsilon=\varUpsilon^{\frac{1}{2}}\circ 
	\varUpsilon^{\frac{1}{2}}$.
	Moreover, $\varUpsilon$, $\Sigma$,  
	$\varUpsilon^{\frac{1}{2}}$, and $\Sigma^{\frac{1}{2}}$ are invertible.
	Hence, we have
	\begin{align}
		\label{e:auxlin}
		(\forall x\in\H)\quad \scal{(\varUpsilon^{-1}-L^*\circ \Sigma\circ L)x}{x}
		&=\|\varUpsilon^{-\frac{1}{2}}x\|^2-\|\Sigma^{\frac{1}{2}}Lx\|^2\nonumber\\
		&=\|\varUpsilon^{-\frac{1}{2}}x\|^2\left(1-
		\dfrac{\|\Sigma^{\frac{1}{2}}L\varUpsilon^{\frac{1}{2}} 
			\varUpsilon^{-\frac{1}{2}}x\|^2}{\|\varUpsilon^{-\frac{1}{2}}x\|^2}\right).
	\end{align}
	Therefore, since $\varUpsilon^{-\frac{1}{2}}$ is a bijection, by denoting 
	$y=\varUpsilon^{-\frac{1}{2}}x$,  \ref{p:linearop1} yields
	\begin{equation}
		\|\Sigma^{\frac{1}{2}}\circ L\circ \varUpsilon^{\frac{1}{2}}\|
		=\sup_{y\in\H}\dfrac{\|\Sigma^{\frac{1}{2}}L\varUpsilon^{\frac{1}{2}} 
			y\|}{\|y\|}\le 1.
	\end{equation}
	The converse clearly holds by using the norm inequality in 
	the right hand side of \eqref{e:auxlin}.
	\ref{p:linearop2}$\Leftrightarrow$\ref{p:linearop3}: Clear from
	$(\Sigma^{\frac{1}{2}}\circ L\circ 
	\varUpsilon^{\frac{1}{2}})^*=\varUpsilon^{\frac{1}{2}}\circ 
	L^*\circ \Sigma^{\frac{1}{2}}$. 
	\ref{p:linearop3}$\Leftrightarrow$\ref{p:linearop4}: It follows 
	from \ref{p:linearop1}$\Leftrightarrow$\ref{p:linearop2} replacing  
	$\Sigma$ by $\varUpsilon$ and $L$ by $L^*$, respectively.
	\ref{p:linearop1}$\Leftrightarrow$\ref{p:linearop5}: For every 
	$(x,u)\in\H\times\G$,
	\begin{align}
		\label{e:mon1}
		\scal{(x,u)}{\vv(x,u)}
		&=\scal{x}{\varUpsilon^{-1}x-L^*u}+\scal{u}{\Sigma^{-1}u-Lx}\nonumber\\
		&=\scal{x}{(\varUpsilon^{-1}-L^*\Sigma L)x}
		+\scal{\Sigma Lx-u}{Lx}
		+\scal{u}{\Sigma^{-1}u-Lx}\nonumber\\
		&=\scal{x}{(\varUpsilon^{-1}-L^*\Sigma L)x}
		+\|\Sigma^{-1}u-Lx\|_{\Sigma}^2
	\end{align}
	and, by symmetry, we analogously obtain
	\begin{equation}
		\label{e:mon2}
		\scal{(x,u)}{\vv(x,u)}=\scal{u}{(\Sigma^{-1}-L\varUpsilon 
			L^*)u}+\|\varUpsilon^{-1}x-L^*u\|_{\varUpsilon}^2.
	\end{equation}
	Hence, it follows from \ref{p:linearop1} and \eqref{e:mon1}
	that $\scal{(x,u)}{\vv(x,u)}\ge \|\Sigma^{-1}u-Lx\|_{\Sigma}^2$.
	Since \ref{p:linearop1} is equivalent to \ref{p:linearop4}, 
	\eqref{e:mon2} yields 
	$\scal{(x,u)}{\vv(x,u)}\ge\|\varUpsilon^{-1}x-L^*u\|_{\varUpsilon}^2$
	and we obtain \eqref{e:vcocomax}. For the converse implication
	it is enough to combine \eqref{e:mon1} with \eqref{e:vcocomax}. 
	
	For the last assertion, note that  \eqref{e:vcocomax} implies,
	for every $(x,u)\in\H\times\G$,
	\begin{equation}
		\label{e:both}
		\begin{cases}
			\scal{(x,u)}{\vv(x,u)}\ge \T\|\varUpsilon^{-1}x-L^*u\|^2\\
			\scal{(x,u)}{\vv(x,u)}\ge \s\|\Sigma^{-1}u-Lx\|^2.
		\end{cases}
	\end{equation}
	By multiplying the first equation in \eqref{e:both} by 
	$\lambda\in\left[0,1\right]$
	and the second by $(1-\lambda)$ and summing up we obtain
	\begin{align}
		\scal{(x,u)}{\vv(x,u)}&\ge \lambda\T\|\varUpsilon^{-1}x-L^*u\|^2
		+(1-\lambda)\s\|\Sigma^{-1}u-Lx\|^2\nonumber\\
		&\ge\min\{\lambda\T,(1-\lambda)\s\}\|\vv(x,u)\|^2.
	\end{align}
	The result follows by noting that $\lambda\mapsto 
	\min\{\lambda\T,(1-\lambda)\s\}$ is maximized at 
	$\lambda^*=\s/(\T+\s)$.
\end{proof}

\section{Convergence of Algorithm~\ref{alg:SDR}}
\label{sec:primaldual}
Denote by $\boldsymbol{M}\colon\H\oplus\G\to 2^{\H\oplus\G}$
the maximally monotone operator \cite[Proposition~2.7]{skew} 
\begin{equation}
	\label{e:defM}
	\boldsymbol{M}\colon (x,u)\mapsto (Ax+L^*u)\times (B^{-1}u-Lx).
\end{equation}
For every strongly monotone self-adjoint linear operators 
$\varUpsilon\colon\H\to\H$ and $\Sigma\colon\G\to\G$, 
consider the real Hilbert space $\HH$ obtained by endowing 
$\H\times\G$ with the inner product 
$\scal{\cdot}{\cdot}_{\boldsymbol{U}}$, where
$\boldsymbol{U}\colon (x,u)\mapsto (\varUpsilon^{-1}x,\Sigma^{-1}u)$.
More precisely,
\begin{equation}
	\label{e:defnorm}
	\scal{\cdot}{\cdot}_{\boldsymbol{U}}\colon 
	\big((x,u),(y,v)\big)\mapsto 
	\scal{x}{\varUpsilon^{-1}y}+\scal{u}{\Sigma^{-1}v},
\end{equation}
and we denote the associated norm by
$\|\cdot\|_{\boldsymbol{U}}
=\sqrt{\scal{\cdot}{\cdot}_{\boldsymbol{U}}}$.
Observe that, since $\varUpsilon$ and $\Sigma$ are strongly 
monotone, the topologies of $\HH$ and $\H\oplus\G$ are 
equivalent.

\begin{prop} \label{prop:opT}
	In the context of Problem~\ref{prob2}, 
	let $\Sigma\colon \G\to\G$ and $\varUpsilon\colon\H\to\H$ be 
	strongly 
	monotone self-adjoint linear operators such that 
	$U=\varUpsilon^{-1}-L^*\Sigma L$
	is monotone, and define $\bm{T} : \HH
	\to \HH$ by
	
	\begin{equation}\label{def:operatorT}
		\bm{T}\colon 
		\begin{pmatrix}
			x\\u
		\end{pmatrix}\mapsto 
		\begin{pmatrix}
			x_+\\u_+
		\end{pmatrix}=
		\begin{pmatrix}
			J_{\varUpsilon A}\big(x- \varUpsilon L^* 
			\Sigma(\id-J_{\Sigma^{-1} 
				B})(Lx+ 
			\Sigma^{-1} u)\big) \\[1mm]
			\Sigma L (x_{+}-x)+  \Sigma(\id-J_{\Sigma^{-1}B})(Lx+ 
			\Sigma^{-1}u)
		\end{pmatrix}.
	\end{equation}
	
	Then, the following hold:
	\begin{enumerate}
		\item \label{prop:TincluM} 
		For every $(x,u)\in\HH$, we have
		\begin{equation}
			\left( \varUpsilon^{-1}(x- x_+),\Sigma^{-1}(u-u_+) \right) \in 
			\bm{M}\big(x_+, u_+-\Sigma L(x_+-x)\big).
		\end{equation}
		
		\item \label{prop:zerosofT} 
		$\fix\bm{T}=\bm{Z}=\zer\boldsymbol{M}$.
		
		\item \label{prop:Tcasicasifirmly} For every $(\hat{x},\hat{u}) 
		\in \bm{Z}$ and $(x,u)\in\HH$ we have
		\begin{align}
			\|\boldsymbol{T}(x,u)-(\hat{x},\hat{u})\|^2_{\boldsymbol{U}}
			&\le \|(x,u)-(\hat{x},\hat{u})\|^2_{\boldsymbol{U}}-
			\|(x,u)-\boldsymbol{T}(x,u)\|^2_{\boldsymbol{U}}\nonumber\\
			&\hspace{3.5cm}+2\scal{ 
				u_+-u}{ L(x_+-x)}.
		\end{align}
	\end{enumerate}
\end{prop}

\begin{proof}
	\ref{prop:TincluM}: 
	From \eqref{def:operatorT} and \cite[Proposition~23.34(iii)]{1} we 
	obtain 
	\begin{align}
		\begin{pmatrix}
			x_+\\ u_+
		\end{pmatrix}
		= \bm{T}
		\begin{pmatrix}
			x\\ u
		\end{pmatrix}
		\:\:
		&\Leftrightarrow\:\: \begin{cases}
			x_+= J_{\varUpsilon A}\big(x- \varUpsilon L^* 
			\Sigma(\id-J_{\Sigma^{-1} 
				B})(Lx+ 
			\Sigma^{-1} u)\big) \\[1mm]
			u_+=\Sigma L (x_{+}-x)+  
			\Sigma(\id-J_{\Sigma^{-1}B})(Lx+ 
			\Sigma^{-1}u)
		\end{cases}\nonumber\\
		&\Leftrightarrow\:\: \begin{cases}
			x_+= J_{\varUpsilon A}\big(x- \varUpsilon L^* (u_+-\Sigma 
			L 
			(x_{+}-x))\big) 
			\\[1mm]
			u_+-\Sigma L (x_{+}-x)=J_{\Sigma B^{-1}}(\Sigma Lx+u)
		\end{cases}\nonumber\\
		&\Leftrightarrow\:\: \begin{cases}
			\varUpsilon^{-1}(x - x_+)-L^*(u_+-\Sigma L(x_+-x )) \in  
			Ax_+ \\[1mm]
			\Sigma^{-1}(u- u_+) + Lx_+ \in  B^{-1}\big(u_+-\Sigma 
			L(x_+-x)\big),
		\end{cases}
		\label{e:auxrem}
	\end{align}
	and the result follows from \eqref{e:defM}.
	\ref{prop:zerosofT}: It follows from \ref{prop:TincluM} and 
	\eqref{Z}
	that
	$\bm{T}(\hat{x},\hat{u})=(\hat{x},\hat{u})$ $\Leftrightarrow$
	$(0,0) \in \bm{M}(\hat{x},\hat{u})$ $\Leftrightarrow$
	$(\hat{x},\hat{u})\in\boldsymbol{Z}$.
	\ref{prop:Tcasicasifirmly}:  Let 
	$(\hat{x},\hat{u})\in\boldsymbol{Z}$.
	It follows from \ref{prop:zerosofT} that 
	$(0,0) \in \bm{M}(\hat{x},\hat{u})$. Hence, \ref{prop:TincluM} 
	and 
	the monotonicity of $\bm{M}$ in $\H\oplus\G$
	yield
	\begin{align*}
		0 &\leq \scal{\varUpsilon^{-1}(x - x_+)}{x_+-\hat{x}}+\scal{ 
			\Sigma^{-1}(u -u_+)}{u_+ -\hat{u}+\Sigma L(x -x_+)}\\
		&\overset{\eqref{e:defnorm}}{=} 
		\scal{(x,u)-(x_+,u_+)}{(x_+,u_+)-(\hat{x},\hat{u})}_{\boldsymbol{U}}
		+\scal{u -u_+}{ L(x -x_+)}\\
		&\overset{\eqref{e:scalU}}{=} 
		\frac{1}{2}\big(\|(x,u)-(\hat{x},\hat{u})\|^2_{\boldsymbol{U}}-
		\|(x,u)-(x_+,u_+)\|^2_{\boldsymbol{U}}
		-\|(x_+,u_+)-(\hat{x},\hat{u})\|^2_{\boldsymbol{U}}\big)\nonumber\\
		&\hspace{5cm}+\scal{u-u_+}{ L(x -x_+)}
	\end{align*}
	and the result follows.
\end{proof}
\begin{rem}
	\label{rem:V}
	\begin{enumerate}
		\item \label{rem:Vi}
		Note that \eqref{def:operatorT} and Algorithm~\ref{alg:SDR}
		yield, for every $n\in\N$, 
		$(x_{n+1},u_{n+1})=(x_{n+},u_{n+})=\boldsymbol{T}(x_n,u_n)$.
		This observation and the properties of $\boldsymbol{T}$ in 
		Proposition~\ref{prop:opT} are crucial for the convergence 
		of Algorithm~\ref{alg:SDR} in Theorem~\ref{teo:GDR} below.
		\item 
		Proposition~\ref{prop:opT}\eqref{prop:Tcasicasifirmly} can be 
		written 
		equivalently as, for every $(\hat{x},\hat{u}) \in 
		\bm{Z}$ and $(x,u)\in\HH$,
		$\|\boldsymbol{T}(x,u)-(\hat{x},\hat{u})\|^2_{\boldsymbol{U}}
		\le \|(x,u)-(\hat{x},\hat{u})\|^2_{\boldsymbol{U}}-
		\|(x,u)-\boldsymbol{T}(x,u)\|^2_{\boldsymbol{V}}$,
		where 
		$\boldsymbol{V}\colon (x,u)\mapsto 
		(\varUpsilon^{-1}x-L^*u,\Sigma^{-1}u-Lx)$.
		Since $\varUpsilon^{-1}-L^*\Sigma L$ is 
		monotone, Proposition~\ref{p:linearop} asserts that
		$\boldsymbol{V}$ is self-adjoint, 
		linear, and cocoercive, but not strongly monotone and, thus,
		$\|\cdot\|^2_{\boldsymbol{V}}$ does not define a norm.
	\end{enumerate}
\end{rem}

\begin{teo} \label{teo:GDR}
	In the context of Problem~\ref{prob2}, let $(x_0,u_0) \in \H\times 
	\G$ and consider the 
	sequence $\big( (x_n,u_n) \big)_{n \in \N}$ defined by the 
	Algorithm~\ref{alg:SDR}. Then, the following assertions hold:
	
	\begin{enumerate}
		\item\label{teo:GDR1} $\sum_{n \geq 1} \|x_{n+1}-x_n\|^2 < \pinf$ and 
		$\sum_{n \geq 1} \|u_{n+1}-u_n\|^2< \pinf$.
		
		\item\label{teo:GDR2} There exists $(\hat{x},\hat{u}) \in 
		\bm{Z}$ such that $(x_n,u_n) \weak (\hat{x},\hat{u})$ in $\H\oplus\G$.
	\end{enumerate}   
\end{teo}
\begin{proof}
	Let ${\boldsymbol{x}}=({x},{u})\in \fix \bm{T}$, for every 
	$n\in\N$, denote by $\boldsymbol{x}_n=(x_n,u_n)$, 
	and fix $n\ge 1$. It follows from 
	Remark~\ref{rem:V}\eqref{rem:Vi}
	that $\boldsymbol{x}_{n+1}=\bm{T} \boldsymbol{x}_n$
	and from Proposition~\ref{prop:opT}\eqref{prop:zerosofT} that 
	${\boldsymbol{x}}\in \bm{Z}$. Therefore, 
	Proposition~\ref{prop:opT}\eqref{prop:Tcasicasifirmly} yields
	
	\begin{align}
		\label{e:primdesig}
		\|\boldsymbol{x}_{n+1}-{\boldsymbol{x}}\|^2_{\boldsymbol{U}}
		&\le \|\boldsymbol{x}_n-{\boldsymbol{x}}\|^2_{\boldsymbol{U}}-
		\|\boldsymbol{x}_n-\boldsymbol{x}_{n+1}\|^2_{\boldsymbol{U}}
		+2\scal{u_{n+1}-u_n}{ L(x_{n+1}-x_n)}.
	\end{align}
	Hence, we deduce from the firm 
	non-expansiveness of $J_{\varUpsilon A}$ in 
	$(\H,\scal{\cdot}{\cdot}_{\varUpsilon^{-1}})$
	\cite[Proposition~23.34(i)]{1} and the monotonicity of 
	$U=\varUpsilon^{-1}-L^*\Sigma L$ that 
	\begin{align}
		&\hspace{-.3cm}\scal{u_{n+1}-u_n}{L(x_{n+1}-x_n)} \nonumber\\
		&\hspace{1.5cm}\overset{\eqref{eq:algoSDR}}{=}
		\scal{\Sigma 
			L(x_{n+1}-x_n)+v_n-\Sigma 
			L(x_n-x_{n-1})-v_{n-1}}{L(x_{n+1}-x_n)}\nonumber\\
		&\hspace{1.5cm}=
		\scal{x_{n+1}-x_n}{	L^*\Sigma L(x_{n+1}-x_n)}+
		\scal{L^*(v_n-v_{n-1})}{x_{n+1}-x_n}\nonumber\\
		&\hspace{2.9cm}-\scal{\Sigma L(x_n-x_{n-1})}{L(x_{n+1}-x_n)}
		\nonumber\\
		&\hspace{1.5cm}=\scal{x_{n+1}-x_n}{
			L^*\Sigma L(x_{n+1}-x_n)}+
		\scal{\varUpsilon^{-1}(x_n-x_{n-1})}{x_{n+1}-x_n}\nonumber\\
		& \hspace*{2.9cm} -\scal{(x_n-\varUpsilon 
			L^*v_n-(x_{n-1}-\varUpsilon 
			L^*v_{n-1}))}{x_{n+1}-x_n}_{\varUpsilon^{-1}}\nonumber\\
		& \hspace*{2.9cm} -\scal{\Sigma 
			L(x_n-x_{n-1})}{L(x_{n+1}-x_n)}\nonumber\\
		&\hspace{1.5cm}\leq \scal{x_{n+1}-x_n}{ 
			L^*\Sigma L(x_{n+1}-x_n)}+
		\scal{\varUpsilon^{-1}(x_n-x_{n-1})}{x_{n+1}-x_n}\nonumber\\
		& \hspace*{2.9cm}-\|x_{n+1}-x_n\|_{\varUpsilon^{-1}}^2 -\scal{ 
			L^*\Sigma L(x_n-x_{n-1})}{x_{n+1}-x_n}\nonumber\\
		&\hspace{1.5cm}= 
		-\|x_{n+1}-x_n\|^2_U+\scal{x_n - x_{n-1} 
		}{x_{n+1}-x_n}_U\nonumber\\
		&\hspace{1.5cm}\overset{\eqref{e:scalU}}{=} 
		-\frac{1}{2}\|x_{n+1}-x_n\|^2_U 
		+\frac{1}{2}\|x_n-x_{n-1}\|^2_U -\frac{1}{2}\|x_{n+1}+x_{n-1}- 
		2x_n\|^2_U\nonumber\\
		&\hspace{1.5cm}\leq -\frac{1}{2}\|x_{n+1}-x_n\|^2_U 
		+\frac{1}{2}\|x_n-x_{n-1}\|^2_U.\label{eq:quasi15}
	\end{align}
	Therefore, it follows from \eqref{e:primdesig} that
	\begin{multline}
		\label{eq:quasi2}
		(\forall n\ge 1)\quad 
		\|\boldsymbol{x}_{n+1}-{\boldsymbol{x}}\|^2_{\boldsymbol{U}}
		+\|x_{n+1}-x_n\|^2_U
		\le \|\boldsymbol{x}_n-{\boldsymbol{x}}\|^2_{\boldsymbol{U}}  +
		\|x_n-x_{n-1}\|^2_U\\-
		\|\boldsymbol{x}_n-\boldsymbol{x}_{n+1}\|^2_{\boldsymbol{U}}.
	\end{multline}
	Thus, \cite[Lemma~3.1]{ipa} asserts that 
	\begin{equation}
		\label{e:converges}
		(\forall {\boldsymbol{x}}\in\boldsymbol{Z})\quad 
		\big( 
		\|\boldsymbol{x}_n-{\boldsymbol{x}}\|^2_{\boldsymbol{U}}+
		\|x_n-x_{n-1}\|^2_U \big)_{n\ge 
			1}\quad\text{converges},
	\end{equation}
	that  
	\begin{equation}
		\label{e:summable}
		\sum_{n\ge 
			1}\|\boldsymbol{x}_{n+1}-\boldsymbol{x}_{n}\|^2_{\boldsymbol{U}}<\pinf,
	\end{equation}
	and \ref{teo:GDR1} follows from \eqref{e:defnorm} and the 
	strong monotonicity of $\varUpsilon^{-1}$ and 
	$\Sigma^{-1}$  \cite[p.266]{RieszNagy}. 
	
	In order to prove \ref{teo:GDR2}, note that, from \ref{teo:GDR1} 
	and the uniform continuity of $U$, we deduce 
	$\|x_n-x_{n-1}\|^2_U\to 0$. Hence, \eqref{e:converges} implies that,
	for every ${\boldsymbol{x}}\in\boldsymbol{Z}$, 
	$(\|\boldsymbol{x}_n-{\boldsymbol{x}}\|^2_{\boldsymbol{U}})_{n\in\N}$
	converges. 
	Now, let $(\overline{x},\overline{u}) \in \HH$ be a 
	weak sequential 
	cluster point of $\big( (x_n,u_n) \big)_{n \in \N}$, say 
	$(x_{k_n},u_{k_n})\weak (\overline{x},\overline{u})$ in $\HH$. It is 
	clear from \eqref{e:defnorm} that we have $x_{k_n}\weakly 
	\overline{x}$ in
	$\H$ and $u_{k_n}\weakly \overline{u}$ in $\G$ and from
	\ref{teo:GDR1}  that $x_{k_n+1}\weakly \overline{x}$
	and $u_{k_n+1}\weak \overline{u}$. Hence, since
	Proposition~\ref{prop:opT}\eqref{prop:TincluM} yields 
	\begin{equation}\label{incluM2}
		\left(\varUpsilon^{-1}(x_{k_n}- x_{k_n+1}),\Sigma^{-1}(u_{k_n}  
		-u_{k_n+1})  
		\right) \in \bm{M}\big(x_{k_n+1}, u_{k_n+1}-\Sigma 
		L(x_{k_n+1}-x_{k_n})\big),
	\end{equation}
	we deduce from \ref{teo:GDR1}, 
	the uniform continuity of $\Sigma L$, $\varUpsilon^{-1}$, and 
	$\Sigma^{-1}$, 
	and
	\cite[Proposition~20.38(ii)]{1}, that $(0,0) \in 
	\bm{M}(\overline{x},\overline{u})$. 
	Therefore, we conclude from  \cite[Lemma~2.47]{1} that 
	there exists $\hat{\boldsymbol{x}}\in\fix\bm{T}$ such that 
	$\boldsymbol{x}_n\weakly \hat{\boldsymbol{x}}$ and the 
	result follows
	from the equivalence of the topologies of $\HH$ 
	and $\H\oplus\G$.
\end{proof}

\begin{rem} 
	\label{rem:SDR}
	\begin{enumerate} 
		\item In  the proof of Theorem~\ref{teo:GDR}, we can also 
		deduce that any weak accumulation point of $((x_n,u_n))_{n\in\N}$
		is in $\boldsymbol{Z}$ by using the points in the graph 
		of $A$ 
		and $B$ obtained from \eqref{e:auxrem} and 
		\cite[Proposition~26.5(i)]{1}.
		\item The method can include summable errors in the 
		computation 
		of resolvents and linear operators, by using standard Quasi-F\'ejer 
		sequences. We prefer to not include this extension 
		for simplicity of our algorithm formulation.
		
		\item Consider the 
		sequences $(v_n)_{n \in \N}$, $(z_n)_{n \in \N}$, 
		$(x_n)_{n \in \N}$, 
		$(u_n)_{n \in \N}$ defined by  Algorithm~\ref{alg:SDR} 
		with starting point $(x_0,u_0)\in\H\times\G$.
		It 
		follows from \eqref{eq:algoSDR} and \cite[Proposition~23.34(iii)]{1}
		that, for every $n\in\N$,
		\begin{align*}
			{v}_{n+1} & = \Sigma 
			(\id-J_{\Sigma^{-1}B})(Lx_{n+1}+\Sigma^{-1}u_{n+1})\\
			& = J_{\Sigma B^{-1}}(\Sigma Lx_{n+1}+u_{n+1})\\
			& = J_{\Sigma B^{-1}}({v}_n+\Sigma L(2x_{n+1}-x_n)),
		\end{align*}
		leading to
		\begin{equation}\label{eq:algoPD}
			(\forall n\in\N)\quad 
			\begin{array}{l}
				\left\lfloor
				\begin{array}{l}
					{x}_{n+1}= J_{\varUpsilon A} ({x}_n -\varUpsilon L^* {v}_n)\\
					{v}_{n+1}=J_{\Sigma B^{-1}}({v}_n+\Sigma 
					L(2{x}_{n+1}-{x}_n)),
				\end{array}
				\right.
			\end{array}
		\end{equation}
		with starting point $(x_0,\Sigma
		(\id-J_{\Sigma^{-1}B})(Lx_{0}+\Sigma^{-1}u_{0}))\in\H\times\G$.
		When $\|\Sigma^{\frac{1}{2}}\circ L\circ 
		\varUpsilon^{\frac{1}{2}}\|<1$, 
		\eqref{eq:algoPD} is equivalent to the proximal point algorithm 
		applied to $\boldsymbol{V}^{-1}\boldsymbol{M}$ in 
		$(\H\times\G,\scal{\cdot}{\cdot}_{\boldsymbol{V}})$, where 
		$\boldsymbol{V}\colon (x,u)\mapsto 
		(\varUpsilon^{-1}x-L^*u,\Sigma^{-1}u-Lx)$ is strongly monotone in 
		view of \cite[Lemma~1]{pockcham}. Moreover, 
		when $\varUpsilon=\tau\id$, $\Sigma=\s\id$, and $\s \T \|L\|^2 < 
		1$, 
		\eqref{eq:algoPD} coincides with 
		the PDS in \eqref{e:condatvu} \cite{cp,condat,yuan,vu}.
		As stated in Remark~\ref{rem:V}, under our assumptions 
		$\boldsymbol{V}$ is no longer strongly monotone and the same 
		approach cannot be used. On the other hand, a generalization of 
		the previous approach is provided in \cite{vu} using the 
		forward-backward 
		splitting in order to allow cocoercive operators in the monotone 
		inclusion when $\boldsymbol{V}$ is strongly monotone. In the 
		optimization context, the inclusion of cocoercive operators allows for 
		convex differentiable functions with $\beta^{-1}-$Lipschitz gradients
		in the objective function and the convergence results are guaranteed 
		under the more restrictive assumption $\s \T \|L\|^2 < 1-\T/2\beta$ 
		\cite[Theorem~3.1]{condat}. Hence, the inclusion of cocoercive 
		operators modifies our monotonicity assumption on $U$ in 
		Algorithm~\ref{alg:SDR} distancing us from our main results. This  
		leads us to consider this extension
		as part of further research.
		
		\item 
		We deduce from \eqref{eq:algoPD} and \eqref{eq:algoSDR} that
		the primal iterates of SDR coincide with those of PDS 
		in \eqref{eq:algoPD} and SDR includes an additional
		inertial step in the dual updates, more precisely,
		\begin{equation}
			\label{e:relcondat}
			(\forall n\in\N)\quad 
			u_{n+1}=\Sigma L(x_{n+1}-x_n)+v_n.
		\end{equation}
		Hence, it follows from 
		Theorem~\ref{teo:GDR}\eqref{teo:GDR1}\&\eqref{teo:GDR2} and 
		the uniform continuity of $\Sigma L$ that
		$v_n\weakly \hat{u}$.
		As a consequence, we obtain the primal-dual weak convergence 
		of \eqref{eq:algoPD} when
		$\|\Sigma^{\frac{1}{2}}\circ L\circ \varUpsilon^{\frac{1}{2}}\| \le 
		1$, which generalizes \cite[Theorem~1]{pockcham}
		and \cite[Theorem~3.3]{condat}, in the case when 
		$\Sigma=\s\id$ 
		and 
		$\varUpsilon=\T\id$, 
		to monotone inclusions and infinite dimensions. 
		
		\item 
		\label{rem:SDRvu}
		By using product space techniques, 
		Algorithm~\ref{alg:SDR} allows us to solve 
		\begin{equation}
			\label{e:vured}
			\text{find }\hat{x}\in\H\:\:\text{such that}\quad  0 \in A 
			\hat{x}+\sum_{i=1}^{m} L_{i}^{*}B_{i} 
			L_{i} \hat{x},
		\end{equation}
		where, for every $i\in\{1,\ldots, m\}$, $\G_i$ is a real Hilbert 
		space, $A\colon \H\to 2^{\H}$ and $B_i\colon \G_i\to 2^{\G_i}$ 
		are maximally monotone, and $L_i\colon \H\to \G_i$ is a linear 
		bounded operator. Indeed, by setting $\G=\oplus_{1\le i\le 
			m}\G_i$, $B\colon (u_i)_{1\le i\le m}\mapsto 
		\times_{i=1}^mB_iu_i$, and 
		$L\colon x\mapsto (L_ix)_{1\le i\le m}$, \eqref{e:vured} is 
		equivalent to \eqref{e:priminc}. Hence, by setting $\Sigma\colon 
		(u_i)_{1\le i\le m}\mapsto  (\Sigma_iu_i)_{1\le i\le m}$, where 
		$(\Sigma_i)_{1\le i\le m}$ are strongly monotone operators, previous 
		remark allows us to write Algorithm~\ref{alg:SDR} as 
		\begin{equation}\label{eq:algoPDvu}
			(\forall n\in\N)\quad 
			\begin{array}{l}
				\left\lfloor
				\begin{array}{l}
					{x}_{n+1}= J_{\varUpsilon A} ({x}_n -\varUpsilon \sum_{i=1}^mL_i^* 
					{v}_{i,n})\\
					{v}_{1,n+1}=J_{\Sigma_1 B_1^{-1}}({v}_{1,n}+\Sigma_1 
					L_1(2{x}_{n+1}-{x}_n))\\
					\hspace{1.2cm}\vdots\\
					{v}_{m,n+1}=J_{\Sigma_m B_m^{-1}}({v}_{m,n}+\Sigma_1 
					L_m(2{x}_{n+1}-{x}_n)),
				\end{array}
				\right.
			\end{array}
		\end{equation}
		and the weak convergence of $(x_n)_{n\in\N}$ to a solution to 
		\eqref{e:vured} is guaranteed by Theorem~\ref{teo:GDR}, 
		assuming that
		\begin{equation}
			\label{e:condvu}
			\varUpsilon^{-1}-\sum_{i=1}^mL_i^*\Sigma_i L_i\quad\text{is 
				monotone.}
		\end{equation}
		Note that 
		\eqref{eq:algoPDvu} has the same structure as the algorithm 
		in
		\cite[Corollary~6.2]{CombVu14} without considering cocoercive 
		operators or relaxation steps, but the convergence is guaranteed 
		under the weaker assumption \eqref{e:condvu}.

		\item Suppose that $\ran L^*=\H$ and that 
		$\varUpsilon=(L^*\Sigma 
		L)^{-1}$. Then, $U=\varUpsilon^{-1}-L^*\Sigma L=0$ and the operator 
		$\TT$ defined in \eqref{def:operatorT} is 
		firmly quasinonexpansive in $\HH$, in view of 
		Proposition~\ref{prop:opT}\eqref{prop:Tcasicasifirmly}  and 
		\eqref{eq:quasi2}.
		We thus generalize \cite[Corollary~3]{svaiter}.
		Observe that, in the particular case when $L=\id$, we 
		have $\varUpsilon=\Sigma^{-1}$ and
		the operator $\TT$ defined in \eqref{def:operatorT} reduces to
		$\bm{T}\colon (x,u)\mapsto 
		\Phi_A^{\varUpsilon}(J_{\varUpsilon  B}(x+ \varUpsilon  
		u)-\varUpsilon u),$
		where
		\begin{equation}\label{def:operatorMinty}
			\Phi_A^{\varUpsilon}\colon \H\mapsto\H\times\H\colon 
			z\mapsto 
			(J_{\varUpsilon A}z,\varUpsilon^{-1}(J_{\varUpsilon A}-\id)z).
		\end{equation}
		In the case when $\varUpsilon=\T\id$, we recover the operator in 
		\cite[Proposition~5.18]{BAR19}, which is inspired by
		\cite{svaiter}. Moreover, note that the 
		inner product $\scal{\cdot}{\cdot}_{\boldsymbol{U}}$ defined in 
		\eqref{e:defnorm} coincides with that in \cite{svaiter} (up to a 
		multiplicative constant). Altogether, Theorem~\ref{teo:GDR} 
		generalizes
		\cite{svaiter} for an arbitrary operator $L$ and non-standard 
		metrics. It also generalizes \cite[Theorem~5.1]{gabay83} from
		variational inequalities to arbitrary monotone inclusions and it provides 
		the weak convergence of shadow sequences $(J_{\tau 
			A}z_n)_{n\in\N}$ (not guaranteed in \cite{gabay83}).
		
		\item Note that, by storing $(Lx_n)_{n\in\N}$, Algorithm~\ref{alg:SDR} 
		only needs to compute $L$ once at each iteration. This observation is 
		important in high dimensional problems in which the computation of 
		$L$ is numerically expensive. 
	\end{enumerate}
\end{rem}
The following result establishes the reduction of
Algorithm~\ref{alg:SDR} to Dou\-glas-Rachford splitting
\cite{eckstein-bertsekas,lions-mercier79}
in the case when $\ran L=\G$.
\begin{prop}\label{prop:GDRtoDR}
	In the context of Problem~\ref{prob2}, assume $\ran L=\G$
	and set $\Sigma=(L\varUpsilon L^*)^{-1}$. 
	Then, 
	Algorithm~\ref{alg:SDR} with starting point $(x_0,u_0)\in\H\times\G$ 
	reduces to the recurrence 
	\begin{equation}
		\label{e:DRS}
		(\forall n\in \N) \quad z_{n+1}=J_{\varUpsilon L^* B L} 
		(2J_{\varUpsilon A}z_n - z_n 
		)+z_n- J_{\varUpsilon A} z_n,
	\end{equation}
	where $z_0=x_0-\varUpsilon 
	L^*\Sigma(\id-J_{\Sigma^{-1}B})(Lx_0+\Sigma^{-1}u_0)$.
\end{prop}
\begin{proof}
	Note that $\ran L=\G$ yields, for every $u\in\G$, 
	$\scal{L\varUpsilon 
		L^*u}{u}\,\ge\,\tau\|L^*u\|^2\ge\tau\alpha^2\|u\|^2$,
	where $\tau>0$ is the strong monotonicity parameter of $\varUpsilon$
	and the existence of $\alpha>0$ is guaranteed by \cite[Fact~2.26]{1}.
	Moreover, it follows 
	from \cite[Proposition~23.34(iii)\&(ii)]{1} 
	that, for every $n\in\N$,
	\begin{align}
		\label{e:desvnp1}
		v_{n+1}&\overset{\eqref{eq:algoSDR}}{=}\Sigma
		(\id-J_{\Sigma^{-1}B})(Lx_{n+1}+\Sigma^{-1}u_{n+1})\nonumber\\
		&=(\Sigma^{-1}+
		B^{-1})^{-1}(Lx_{n+1}+\Sigma^{-1}u_{n+1})\nonumber\\
		&\overset{\eqref{eq:algoSDR}}{=}
		(\Sigma^{-1}+
		B^{-1})^{-1}(L(2x_{n+1}-x_n)+\Sigma^{-1}v_{n})\nonumber\\
		&=(L\varUpsilon L^*+
		B^{-1})^{-1}L(2x_{n+1}-x_n+\varUpsilon L^*v_{n}),
	\end{align}
	where the last equality follows from $\Sigma^{-1}=L\varUpsilon 
	L^*$.
	On the other hand, \cite[Proposition~23.34(iii)]{1} yields
	\begin{align}
		\label{e:rescomp}
		J_{\varUpsilon 
			L^*BL}&=\varUpsilon^{\frac{1}{2}}J_{\varUpsilon^{\frac{1}{2}}
			L^*BL\varUpsilon^{\frac{1}{2}}}\varUpsilon^{-\frac{1}{2}}\nonumber\\
		&=\varUpsilon^{\frac{1}{2}}(\id-\varUpsilon^{\frac{1}{2}}L^*
		(L\varUpsilon L^*+ 
		B^{-1})^{-1}L\varUpsilon^{\frac{1}{2}})\varUpsilon^{-\frac{1}{2}}
		\nonumber\\
		&=\id-\varUpsilon L^*(L\varUpsilon L^*+ 
		B^{-1})^{-1}L,
	\end{align}
	where the second equality follows from 
	\cite[Proposition~23.25(ii)]{1}
	since 
	$(L\varUpsilon^{\frac{1}{2}})(L\varUpsilon^{\frac{1}{2}})^*=L\varUpsilon
	L^*$ is invertible. Hence, we have
	\begin{align*}
		z_{n+1}&\overset{\eqref{eq:algoSDR}}{=}x_{n+1}-\varUpsilon 
		L^*v_{n+1}\\
		&\overset{\eqref{e:desvnp1}}{=}x_{n+1} - \varUpsilon L^* 
		(L\varUpsilon L^*+ 
		B^{-1})^{-1}L( 
		2x_{n+1}-x_n+\varUpsilon 
		L^*v_n)\\
		&\overset{\eqref{eq:algoSDR}}{=} \left(\id-  \varUpsilon L^* 
		(L\varUpsilon L^*+ 
		B^{-1})^{-1}L\right)(2J_{\varUpsilon A} - \id 
		)z_n+(\id- J_{\varUpsilon A})z_n\\
		&\overset{\eqref{e:rescomp}}{=}J_{\varUpsilon L^* B L} 
		(2J_{\varUpsilon A} - \id )z_n+(\id- 
		J_{\varUpsilon A})z_n
	\end{align*}
	and $z_0$ is obtained from 
	\eqref{eq:algoSDR}.
\end{proof}

\begin{rem}
	Note that $\Sigma=(L\varUpsilon L^*)^{-1}$ is equivalent to 
	$\Sigma^{-1}-L\varUpsilon L^*=0$ and, hence,  
	$\varUpsilon^{-1}-L^*\Sigma L$
	is monotone in view of Proposition~\ref{p:linearop}. Therefore,
	Proposition~\ref{prop:GDRtoDR} and Theorem~\ref{teo:GDR}
	provide the weak convergence of the non-standard 
	metric version of DRS in 
	\eqref{e:DRS} when $\ran L=\G$.
	This also extends the convergence result in \cite{svaiter}.
\end{rem}

\section{Split ADMM}
\label{sec:SADMM}
In this section we study the numerical approximation of the 
following convex optimization problem. 
\begin{pro}\label{prob:optipd}
	Let $\H$, $\G$, and $\K$ be real Hilbert spaces. Let 
	$g\in\Gamma_0( \K) $, let $f\in\Gamma_0 (\H)$, and let $T: \K 
	\to \G$ and $K:\G \to \H$ be 
	non-zero bounded linear operators  such that $\ran T^*\cap\dom 
	g^*\ne\varnothing$.
	Consider the following optimization problem 
	\begin{equation}\label{eq:primalminproblem}\tag{$P$}
		\min_{y \in \K} \big(g(y)+ f(K T y)\big)
	\end{equation}
	together with the associated Fenchel-Rockafellar dual
	\begin{equation}\label{eq:dualminproblem}\tag{$D$}
		\min_{x \in \H} \big(f^*(x)+g^*(- T^* K^* x)\big).
	\end{equation}
	Moreover, consider the following Fenchel-Rockafellar dual 
	problem 
	associated to 
	\eqref{eq:dualminproblem}
	\begin{equation}\label{eq:dualminproblemdual}\tag{$P^*$}
		\min_{u \in \G} \big((g^*\circ -T^*)^* (u) + f(-K u)\big).
	\end{equation}
	We denote by
	$S_{P}$, $S_{D}$, and $S_{P^*}$ the set of solutions to 
	\eqref{eq:primalminproblem}, \eqref{eq:dualminproblem}, and 
	\eqref{eq:dualminproblemdual}, respectively.
\end{pro}
In the particular case when $K=\id$, Problem~\ref{prob:optipd}
is also considered in 
\cite{ecksteinthesis,gabay83,patrinos2020,yin2016} and ADMM is 
derived in \cite{gabay83} by applying DRS 
to the first order optimality conditions of 
\eqref{eq:dualminproblem}, with 
$A= \partial f^*$ and 
$B= \partial (g^* \circ (- T^*K^*))$. We generalize this procedure
by applying Algorithm~\ref{alg:SDR} to \eqref{eq:dualminproblem}
with $A= \partial f^*$,
$B= \partial (g^* \circ (- T^*))$, and $L=K^*$. We thus obtain the 
Split-ADMM (SADMM), which splits $K$ from $T$. We now 
provide an example in which this new formulation is relevant.
\begin{ejem}
	\label{ex:sadmm}
	Let $A$ and $M$ be $n\times N$ and $m\times N$ real matrices, 
	respectively,
	let $b\in\R^n$, let $\phi\in\Gamma_0(\R^m)$, let 
	$h\in\Gamma_0(\R^n)$, and consider the 
	optimization problem
	\begin{equation}
		\label{e:datafid}
		\min_{y\in \R^N}h(Ay-b)+\phi(My).
	\end{equation} 
	This problem arises in image and signal restoration and 
	denoising 
	\cite{ChambLions97,Pustelnik2019,Daube04,Liu2014,Pang2017,Sha2020}.
	If $M$ is symmetric and 
	positive definite, as in graph Laplacian regularization (see, e.g., 
	\cite[Section~II.B]{Liu2014} and \cite{Pang2017,Sha2020} for 
	alternative regularizations),
	there exist $P$ unitary and $D$ diagonal such that $M=PDP^{\top}$.
	Therefore, by setting $\eta\in\left]0,1\right[$, $K=PD^{\eta}P^{\top}$,
	$T=PD^{1-\eta}P^{\top}$, $g=\phi$, and $f=h(A\cdot-b)$,
	\eqref{e:datafid} is a particular instance of 
	\eqref{eq:primalminproblem}. In some instances, the resolvent 
	computation of $\partial (g^*\circ -T^*)$ is simpler to solve than that of 
	$\partial (g^*\circ -T^*K^*)$ when $\eta\sim 1$, since 
	$D^{1-\eta}\sim\id$. The numerical advantage of this 
	approach is 
	illustrated in an academical example in Section~\ref{sec:toy}. 
	
	Other potential applications arise naturally when 
	$y=\Phi z$, where $z$ denotes frequencies or wavelet coefficients
	of an image $y$ and $\Phi$ is a frame or unitary linear operator 
	allowing to pass from frequencies to images. Therefore,
	\eqref{e:datafid} is a particular case of \eqref{eq:primalminproblem}
	when $f=h(\cdot-b)$, $g=\phi\circ M\circ \Phi$, $K=A$, and $T=\Phi$.
	The properties of $T$ in this case also make preferable to split $T$
	from $K$. 
\end{ejem}
First we 
provide some existence 
results and connections between problems 
\eqref{eq:primalminproblem}, \eqref{eq:dualminproblem}, and 
\eqref{eq:dualminproblemdual}.
\begin{prop}\label{prop:relacionproblemasduales}
	In the context of Problem~\ref{prob:optipd}, 
	consider the inclusion
	\begin{equation} \label{eq:z_opt1}
		\text{find}\quad (\hat{x}, \hat{u}) \in \H \times \G \quad \text{such that} \quad\begin{cases}
			0 \in \partial f^* (\hat{x})+K\hat{u}\\
			0 \in  \partial(g^*\circ-T^*)^*(\hat{u})-K^*\hat{x}.
		\end{cases}
	\end{equation}
	\begin{enumerate}
		\item Suppose that there exists 
		$\hat{y} \in S_P$ and that one of the following 
		assertions hold:
		\begin{enumerate}
			\item \label{prop:relacionproblemasdualesa} 
			$0 \in \partial g (\hat{y}) + T^*K^* \partial f( K T  
			\hat{y})$.
			\item \label{prop:relacionproblemasduales1}  
			$0 \in \sri(\dom f - KT \dom g)$.
		\end{enumerate}
		Then, there exists $\hat{x}\in S_D$ such that
		$(\hat{x},-T\hat{y})$ is a solution to 
		\eqref{eq:z_opt1}. 
		
		\item Suppose that there exists $\hat{x} \in 
		S_D$ and that one of the following 
		assertions hold:
		\begin{enumerate}
			\item \label{prop:relacionproblemasdualesb} 
			$0 \in \partial f^* (\hat{x}) - KT \partial g^*(-T^*K^*  
			\hat{x})$.
			\item \label{prop:relacionproblemasduales2} 
			$0 \in \sri(\dom g^*-T^*K^*\dom 
			f^*)$.
			\item \label{prop:relacionproblemasduales3} 
			$0 \in \sri(\dom (g^*\circ-T^*)-K^*\dom f^*)$ and $0\in\sri (\dom 
			g^*-\ran\, T^*)$.
		\end{enumerate}
		Then, there exists $\hat{y}\in S_P$ such that $(\hat{x},-T\hat{y})$ is a 
		solution to \eqref{eq:z_opt1}.
		
		\item 
		\label{prop:relacionproblemasduales3real} 
		Suppose that there exists $(\hat{x},\hat{u})$ solution to 
		\eqref{eq:z_opt1} and that  $0 \in \sri (\dom g^* - \ran T^*)$. Then, 
		$(\hat{x},\hat{u})\in S_D\times S_{P^*}$ and
		there exists $\hat{y}\in S_P$ such that $\hat{u}=-T\hat{y}$.
	\end{enumerate}
\end{prop}

\begin{proof}
	\ref{prop:relacionproblemasdualesa}:
	Let $\hat{x}\in\partial f(KT\hat{y})$ be such that 
	$ 0 \in  \partial g(\hat{y})+T^*K^*\hat{x}$. Hence, it follows from 
	\cite[Corollary~16.30]{1} that
	\begin{equation}
		\begin{cases}
			\label{e:auxexist}
			0 \in \partial f^* (\hat{x})-KT\hat{y}\\
			0 \in  \partial g(\hat{y})+T^*K^*\hat{x},
		\end{cases}
	\end{equation}
	and 
	\cite[Proposition~2.8(i)]{skew} implies $(\hat{y},\hat{x})\in 
	S_P\times 
	S_D$. By defining $\hat{u}=-T\hat{y}$, we obtain 
	$0 \in \partial f^* 
	(\hat{x})+K{\hat{u}}$. Moreover, $\ran T^*\cap\dom 
	g^*\ne\varnothing$ yields $g^*\circ(-T^*)\in\Gamma_0(\K)$ and 
	$-T(\partial g^*)(-T^*)\subset \partial(g^*\circ-T^*)$
	in view of \cite[Proposition~16.6(ii)]{1}.
	Hence, we deduce from 
	\cite[Corollary~16.30]{1} and \eqref{e:auxexist}
	that 
	\begin{align}
		-T^*K^* \hat{x} \in \partial g (\hat{y}) \quad &\Leftrightarrow \quad 
		\hat{y} \in \partial g^*(-T^*K^*\hat{x})\nonumber\\
		&\Rightarrow\quad  \hat{u}=-T\hat{y} \in  -T \partial 
		g^*(-T^*K^*\hat{x})\nonumber\\
		&\Rightarrow\quad  \hat{u} \in 
		\partial(g^*\circ-T^*)(K^*\hat{x})\label{eq:relacionproblemasduales12}\\
		&\Leftrightarrow\quad K^*\hat{x} \in  
		\partial(g^*\circ-T^*)^*(\hat{u})\nonumber\\
		&\Leftrightarrow\quad 0 \in  
		\partial(g^*\circ-T^*)^*(\hat{u})-K^*\hat{x}\label{eq:relacionproblemasduales13}.
	\end{align}
	Therefore, $(\hat{x},-T\hat{y})$ is a solution to 
	\eqref{eq:z_opt1}.
	
	\ref{prop:relacionproblemasduales1}: By 
	\cite[Theorem~16.3 \& Theorem~16.47(i)]{1}, $0 \in 
	\partial(g+f\circ 
	K
	T)(\hat{y})=\partial g 
	(\hat{y}) + T^*K^* \partial f( K T  \hat{y})$. The result follows from 
	\ref{prop:relacionproblemasdualesa}.
	
	\ref{prop:relacionproblemasdualesb}: Since, by taking 
	$\hat{y}\in  \partial g^*(-T^*K^* \hat{x})$ such that $0\in\partial 
	f^*(\hat{x})-KT\hat{y}$, we obtain 
	\eqref{e:auxexist}, the argument is analogous to that 
	in \ref{prop:relacionproblemasdualesa}.
	
	\ref{prop:relacionproblemasduales2}: By 
	\cite[Theorem~16.3 \& Theorem~16.47(i)]{1}, 
	$0 \in\partial(f^*+g^*\circ(-T^*K^*))(\hat{x})= \partial 
	f^*(\hat{x})-KT(\partial g^*)(-T^*K^* \hat{x})$.
	The result hence follows from \ref{prop:relacionproblemasdualesb}.
	
	\ref{prop:relacionproblemasduales3}: By 
	\cite[Theorem~16.3 \& Theorem~16.47(i)]{1}, 
	$0 \in  \partial f^*  (\hat{x})+K\partial (g^* \circ -T^*)(K^*\hat{x})$. 
	Moreover $0\in\sri (\dom g^*-\ran\, T^*)$ and \cite[Theorem~16.47]{1} 
	imply $0 \in \partial f^*  (\hat{x})-KT(\partial g^*) (-T^*K^*\hat{x})$.
	The result hence follows from \ref{prop:relacionproblemasdualesb}.
	
	\ref{prop:relacionproblemasduales3real}: It follows from the 
	second 
	inclusion of
	\eqref{eq:z_opt1} and  \cite[Theorem~16.47]{1}  
	that $\hat{u}\in\partial (g^*\circ(-T^*))(K^*\hat{x})=-T\partial 
	g^*(-T^*K^*\hat{x})$. Hence, there exists $\hat{y}\in\partial 
	g^*(-T^*K^*\hat{x})$ such that $\hat{u}=-T\hat{y}$, which yields
	$0\in \partial g(\hat{y})+T^*K^*\hat{x}$. Therefore, by combining 
	$\hat{u}=-T\hat{y}$ with the first inclusion of \eqref{eq:z_opt1},
	we deduce \eqref{e:auxexist} and the result follows from 
	\cite[Proposition~2.8(i)]{skew}.
	
\end{proof}
\begin{rem}
	In the context of 
	Proposition~\ref{prop:relacionproblemasduales}\eqref{prop:relacionproblemasduales3real}
	we obtain the existence of $\hat{y}\in S_P$ such that 
	$(\hat{x},\hat{y})$
	satisfies \eqref{e:auxexist}. If we additionally assume that 
	$\ran T$ is closed, the second equation in \eqref{e:auxexist}
	implies that
	$\hat{y}\in 
	\arg\min_{Ty=-\hat{u}}g(y)$.
	We thus recover the results in \cite[Lemma 2]{yin2016},
	obtained when $K=-\id$.
\end{rem}

\begin{algo}[Split-Alternating Direction Method of Multipliers 
	(SADMM)]
	\label{algo:43}
	In the context of 
	Problem~\ref{prob:optipd}, 
	let $\Sigma\colon \G\to\G$ and $\varUpsilon\colon\H\to\H$ be 
	strongly 
	monotone self-adjoint linear operators such that 
	$\Sigma^{-1}-K^*\varUpsilon K$ is monotone, let $p_0 \in \K$, 
	and let 
	$(q_0,x_0) \in 
	\H\times \H $. 
	Consider, the sequences defined by the recurrence
	\begin{equation}
		\label{eq:admmgeneralized1}
		(\forall n\in\N)\quad 
		\begin{array}{l}
			\left\lfloor
			\begin{array}{l}
				y_{n}=x_n+\varUpsilon (K T p_{n} - q_{n})\\[2mm]
				p_{n+1} \in \argm{p \in \K} \big(g (p)+\frac{1}{2}\|T p - 
				(Tp_n- \Sigma K^* y_{n})\|_{\Sigma^{-1}}^2\big)\\
				q_{n+1} = 
				\prox^{\varUpsilon}_{f}(\varUpsilon^{-1}x_n+KTp_{n+1})\\[2mm]
				x_{n+1}=x_n + \varUpsilon  (K T p_{n+1} - 
				q_{n+1}).
			\end{array}
			\right.
		\end{array}
	\end{equation}
\end{algo}
Observe that the existence and 
uniqueness of solutions to the convex optimization problem of the 
second step of \eqref{eq:admmgeneralized1} is not guaranteed 
without further hypotheses. 
The following result 
provides sufficient conditions for the existence
of solutions to the optimization problem in 
\eqref{eq:admmgeneralized1}, the 
equivalence between the sequences generated by 
Algorithm~\ref{alg:SDR} and Algorithm~\ref{algo:43}, 
and the weak convergence of SADMM.

\begin{teo} \label{teo:admmgeneralized} In the context of 
	Problem~\ref{prob:optipd}, suppose 
	that there exists a solution to \eqref{eq:z_opt1}, set 
	\begin{equation}
		\label{e:SADMMdefsops}
		A= \partial f^*,\quad 
		B= \partial (g^* \circ (- T^*)),\quad\text{and}\quad 
		L=K^*,
	\end{equation}
	and assume 
	that $0 \in \sri (\dom g^* - \ran T^*)$.
	Then, $(p_n)_{n\in\N}$ defined in 
	\eqref{eq:admmgeneralized1} exists
	and the following statements hold. 
	\begin{enumerate}
		\item
		\label{teo:admmgeneralizedi}(SDR reduces to SADMM)
		Let $(\tilde{x}_n)_{n \in \N}$, $(\tilde{u}_n)_{n \in 
			\N}$, and $(\tilde{v}_n)_{n \in \N}$ be the sequences generated 
		by Algorithm~\ref{alg:SDR} and set 
		\begin{align}\label{eq:pqSDRtoADMM}
			(\forall n \in \N) \quad \begin{cases} 
				\tilde{p}_{n+1} \in  T^{-1}(-\tilde{v}_n)\\
				\tilde{q}_{n+1}= 
				\varUpsilon^{-1}(\tilde{x}_n-\tilde{x}_{n+1}-\varUpsilon K 
				\tilde{v}_{n}).
			\end{cases}
		\end{align}
		Moreover, set $p_1\in\K$ such that
		$T p_1=T \tilde{p}_1$,  and
		$q_1=\tilde{q}_1$, 
		$x_1=\tilde{x}_1$. Then, sequences
		$(p_n)_{n\ge 1}$, 
		$(q_n)_{n\ge 1}$, and $(x_n)_{n\ge 1}$ generated 
		by Algorithm~\ref{algo:43} satisfy,
		for every $n \geq 1$,
		$T\tilde{p}_{n}=Tp_{n}$, $\tilde{q}_n=q_{n}$, and 
		$\tilde{x}_n=x_{n}$. 
		
		\item 
		\label{teo:admmgeneralizedii} 
		(SADMM reduces to SDR) Let $(p_n)_{n\ge 1}$, $(q_n)_{n\ge 
			1}$, 
		and $(x_n)_{n\ge 1}$ be sequences
		generated by Algorithm~\ref{algo:43} and define
		\begin{align}
			\label{eq:defproof}
			(\forall n \in \N)\quad  u_{n+1}=  \Sigma 
			K^*(x_{n+1}-x_n)-Tp_{n+1}.
		\end{align}
		Moreover, set $\tilde{x}_0=x_1$, 
		$\tilde{u}_0=u_1$, and let $(\tilde{x}_n)_{n \in \N}$ and 
		$(\tilde{u}_n)_{n \in 
			\N}$ be the sequences generated by 
		Algorithm~\ref{alg:SDR}.
		Then, for all $n\in\N$,
		$\tilde{x}_n=x_{n+1}$ and $\tilde{u}_n=u_{n+1}$.
		
		\item  
		\label{teo:admmgeneralizediii}
		Let
		$(p_n)_{n\in\N}$, $(q_n)_{n\in\N}$, 
		and $(x_n)_{n\in\N}$ be sequences
		generated by Algorithm~\ref{algo:43}.
		Then, the following hold:
		\begin{enumerate}
			\item 
			\label{teo:admmgeneralizediiia}
			There exists $(\hat{y},\hat{x},\hat{u} )\in S_P\times 
			S_D\times S_{P^*}$ such that 
			$(x_{n}, -T p_n,q_n) \weak (\hat{x},\hat{u},-K\hat{u})$ and 
			$\hat{u}= - T \hat{y}$. 
			\item 
			\label{teo:admmgeneralizediiib}
			Suppose that $\ran T^*=\K$. Then, there 
			exists 
			$\hat{y}\in S_P$ such that $p_n\weakly \hat{y}$.	
		\end{enumerate}
	\end{enumerate}
\end{teo}
\begin{proof} 
	Note that $g^*\circ -T^*\in\Gamma_0(\G)$, that 
	\cite[Corollary~16.53]{1} yields $B=-T\circ (\partial g^*)\circ -T^*$, 
	and that 
	$J_{\Sigma^{-1}B}=(\id-\Sigma^{-1}T(\partial g^*)(-T^*))^{-1}$.
	Therefore, it 
	follows from \cite[Corollary~16.30]{1} that
	\begin{align}
		(\forall (u,y)\in\G^2)\quad 
		y=J_{\Sigma^{-1}B}u\quad 
		&\Leftrightarrow\quad  (u-y)\in-\Sigma^{-1}T
		\partial g^*(-T^* y)\nonumber\\
		&\Leftrightarrow\quad(\exists p\in\K)\:\: 
		\begin{cases}
			y=u+\Sigma^{-1}Tp\\
			p\in \partial g^*(-T^* y)
		\end{cases}\nonumber\\
		&\Leftrightarrow\quad(\exists p\in\K)\:\: 
		\begin{cases}
			y=u+\Sigma^{-1}Tp\\
			0\in \partial g(p)+T^* y
		\end{cases}\nonumber\\
		&\Leftrightarrow\quad(\exists p\in\K)\:\: 
		\begin{cases}
			y=u+\Sigma^{-1}Tp\\
			p\in S(u),
		\end{cases}
	\end{align}
	where $S\colon u\mapsto \arg\min 
	(g+\frac{1}{2}\|T\cdot+\Sigma 
	u\|_{\Sigma^{-1}}^2)$ and last equivalence follows from 
	\cite[Theorem~16.3]{1} and simple gradient computations. We 
	conclude $\dom S=\G$, 
	$\prox_{g^*\circ -T^*}^{\Sigma}=\id+\Sigma^{-1}TS$, and, therefore,
	\begin{equation}
		\label{e:proxcalc}
		\Sigma(\id-J_{\Sigma^{-1}B})=\Sigma\big(\id-\prox_{g^*\circ 
			-T^*}^{\Sigma}\big)=-TS.
	\end{equation}
	Thus, the optimization problem in \eqref{eq:admmgeneralized1} is 
	equivalent to
	\begin{equation}
		\label{e:seqpadmm}
		(\forall n\in\N)\quad p_{n+1}\in S\big(K^*(x_n+\varUpsilon 
		(KTp_n-q_n))-\Sigma^{-1}Tp_n\big)
	\end{equation}
	and, hence, sequence 
	$(p_n)_{n\in\N}$ exists.
	
	\ref{teo:admmgeneralizedi}: It follows from \eqref{eq:pqSDRtoADMM}, 
	\eqref{eq:algoSDR}, \eqref{e:proxcalc}, and \eqref{e:SADMMdefsops} 
	that, for every $n\in\N$,
	$T\tilde{p}_{n+1}=-\tilde{v}_n=TS(K^*\tilde{x}_n+\Sigma^{-1}
	\tilde{u}_n)$ and, thus,
	$\tilde{u}_{n+1}=\Sigma K^*\varUpsilon (KT 
	\tilde{p}_{n+1}-\tilde{q}_{n+1})-T\tilde{p}_{n+1}$.
	Therefore, we have
	\begin{equation}
		(\forall n\ge 1)\quad T\tilde{p}_{n+1}=TS(K^*(\tilde{x}_n+\varUpsilon 
		(KT\tilde{p}_n-\tilde{q}_n))-\Sigma^{-1}T\tilde{p}_n).
	\end{equation}
	In addition, from \eqref{eq:algoSDR}, \eqref{e:SADMMdefsops}, and 
	\eqref{e:Moreau_nonsme} 
	we have, for every $n\in\N$, $\tilde{x}_{n+1}=\tilde{x}_n+\varUpsilon 
	KT\tilde{p}_{n+1}-\varUpsilon\prox_f^{\varUpsilon}(
	\varUpsilon^{-1}\tilde{x}_n+KT\tilde{p}_{n+1})$ and, thus, 
	\eqref{eq:pqSDRtoADMM} yields 
	$\tilde{q}_{n+1}=\prox_f^{\varUpsilon}(
	\varUpsilon^{-1}\tilde{x}_n+KT\tilde{p}_{n+1})$.
	Altogether, we deduce
	\begin{equation}
		\label{eq:sdrtoadmm}
		(\forall n\ge 1)\quad 
		\begin{array}{l}
			\left\lfloor
			\begin{array}{l}
				T\tilde{p}_{n+1} =TS(K^*(\tilde{x}_n+\varUpsilon 
				(KT\tilde{p}_n-\tilde{q}_n))-\Sigma^{-1}T\tilde{p}_n)\\[1mm]
				\tilde{q}_{n+1} = 
				\prox^{\varUpsilon}_{f}(\varUpsilon^{-1}\tilde{x}_n+KT\tilde{p}_{n+1})\\[1mm]
				\tilde{x}_{n+1}=\tilde{x}_n +\varUpsilon  (K T \tilde{p}_{n+1} - 
				\tilde{q}_{n+1})
			\end{array}
			\right.
		\end{array}
	\end{equation}
	and the result follows from \eqref{e:seqpadmm}, $x_1=\tilde{x}_1$, 
	$q_1=\tilde{q}_1$, and $Tp_{1}=T\tilde{p}_1$.
	
	\ref{teo:admmgeneralizedii}:  
	Define
	\begin{equation}
		\label{e:defvz}
		(\forall n\in\N)\quad
		\begin{cases}
			v_n = -T p_{n+1}\\
			z_n=x_n+\varUpsilon KT p_{n+1}
		\end{cases} 
	\end{equation}
	and fix $n\ge 1$.
	Hence, we have
	\begin{align}\label{eq:qn}
		q_{n+1} \overset{\eqref{eq:admmgeneralized1}}{=} 
		\prox^{\varUpsilon}_{f}&(\varUpsilon^{-1}x_n+KTp_{n+1}) 
		\nonumber\\
		& \Leftrightarrow\quad x_n+\varUpsilon ( KTp_{n+1}-q_{n+1})
		\overset{\eqref{e:Moreau_nonsme}}{=} 
		\prox^{\varUpsilon^{-1}}_{f^*}(x_n+\varUpsilon KT 
		p_{n+1})\nonumber\\
		& \Leftrightarrow\quad x_{n+1}
		\overset{\eqref{e:SADMMdefsops}}{=}  
		J_{\varUpsilon A}z_n.
	\end{align}
	Moreover, from \eqref{e:seqpadmm}, \eqref{eq:defproof},
	and \eqref{eq:admmgeneralized1},
	we obtain
	$p_{n+1}\in S(K^*x_n+\Sigma^{-1}u_n)$. 
	Hence,   \eqref{e:proxcalc}, \eqref{e:SADMMdefsops},
	and \eqref{e:defvz} yield
	$v_n=\Sigma (\id -J_{\Sigma^{-1}B})(Lx_n+\Sigma^{-1}u_n)$.
	Altogether, from \eqref{eq:defproof} we recover the recurrence in 
	Algorithm~\ref{alg:SDR} 
	shifted by one iteration and, by setting $\tilde{x}_0=x_1 $ and 
	$\tilde{u}_0=u_1$ the result follows.
	
	\ref{teo:admmgeneralizediiia}.  Set 
	$(u_{n})_{n\ge 1}$ 
	via \eqref{eq:defproof} and define, for every $n\in\N$,
	$\tilde{x}_n=x_{n+1}$ and $\tilde{u}_{n}=u_{n+1}$. 
	Then, \ref{teo:admmgeneralizedii}
	asserts that
	$(\tilde{x}_n)_{n \in \N}$ and $(\tilde{u}_n)_{n \in 
		\N}$ are the sequences generated by 
	Algorithm~\ref{alg:SDR} with the operators defined in 
	\eqref{e:SADMMdefsops}.
	Note that $A=\partial g^*$ and $B=\partial (g^* \circ (- T^*))$ 
	are maximally monotone \cite[Theorem 20.25]{1}
	and that the set 
	$\bm{Z}$ defined in \eqref{Z} is the primal-dual solution set to 
	the inclusion 
	\eqref{eq:z_opt1}, which is non-empty by hypothesis. Then, by 
	Theorem~\ref{teo:GDR}\eqref{teo:GDR2}, there exists some 
	$(\hat{x},\hat{u})$ solution to \eqref{eq:z_opt1} such that 
	$(\tilde{x}_{n},\tilde{u}_n)=(x_{n+1}, u_{n+1}) \weak (\hat{x},\hat{u})$. 
	Moreover, Theorem~\ref{teo:GDR}\eqref{teo:GDR1} yields 
	\begin{equation}
		\label{e:to0}
		x_{n+1}-x_{n} \to 0, 
	\end{equation}
	and, thus, \eqref{eq:defproof} yields 
	$-Tp_{n+1}=u_{n+1}-\Sigma 
	K^*(x_{n+1}-x_n) \weak 
	\hat{u}$. 
	Hence, since
	\eqref{eq:admmgeneralized1} yields, for every $n \in \N$, 
	$q_{n+1}=\varUpsilon^{-1}(x_n-x_{n+1}) + KT 
	p_{n+1}$, the weak continuity of $K$ and \eqref{e:to0}
	imply $q_n \weak -K\hat{u}$. 
	We conclude that $(x_n, -T p_n, q_n) \weak 
	(\hat{x},\hat{u},-K\hat{u})$. The result follows from
	Proposition~\ref{prop:relacionproblemasduales}\eqref{prop:relacionproblemasduales3real}.
	
	\ref{teo:admmgeneralizediiib}. By \ref{teo:admmgeneralizediiia}, 
	there exists $\hat{y} \in S_P$ such that $Tp_n \weak T\hat{y}$. 
	Thus, 
	for every $z\in\K$, there exists $w\in\G$ such that $z=T^*w$, 
	which yields $\scal{z}{p_n-\hat{y}}=\scal{w}{Tp_n-T\hat{y}}\to 0$
	and, hence, $p_n \weak \hat{y}$. This concludes the proof.
\end{proof}
\begin{rem}
	\label{r:remSADMM}
	\begin{enumerate}
		\item\label{r:remSADMM1}
		Note that the existence of a sequence $(p_n)_{n\in\N}$
		is guaranteed without any further assumption than $0 \in \sri (\dom g^* 
		- \ran T^*)$. 
		This result is weaker than strong monotonicity or full range 
		assumptions
		made in \cite{BrediesSun17,gabay83} and improves 
		\cite{ecksteinadmm2015}, in which this existence
		is assumed. Note that, even if there could exist a continuum of 
		solutions to the optimization problem in \eqref{eq:admmgeneralized1},
		the image through $T$ is unique, in view of \eqref{e:seqpadmm} and
		\eqref{e:proxcalc}.
		
		\item In the case when $K=\id$, 
		Theorem~\ref{teo:admmgeneralized}\eqref{teo:admmgeneralizedi}
		recovers the reduction of 
		DRS when $A=\partial f^*$ and $B=\partial 
		(g^*\circ 
		(-T^*))$ to ADMM and the convergence is guaranteed under weaker 
		conditions than the strong monotonicity and full range assumptions 
		used in \cite[Section~5.1]{gabay83}.
		Under the assumption $\ker T=\{0\}$, this result is obtained in 
		\cite[Theorem 3.2]{moursi2019}.
		
		\item Suppose that $K=\id$. Observe 
		that, given the sequence $(\tilde{v}_n)_{n\in\N}$ 
		generated by SDR, 
		Theorem~\ref{teo:admmgeneralized}\eqref{teo:admmgeneralizedii}
		asserts that any sequence 
		$(p_n)_{n\in\N}$ satisfying $-Tp_{n+1}=\tilde{v}_n$ allows the 
		convergence of ADMM and its equivalence with DRS applied to the 
		dual problem \eqref{eq:dualminproblem}. The equivalence of ADMM
		with respect to DRS applied to the primal \eqref{eq:primalminproblem}
		is studied in \cite{patrinos2020,yin2016}.
		
		\item In the case when $K=\id$, 
		Theorem~\ref{teo:admmgeneralized}\eqref{teo:admmgeneralizedii}
		provides the reduction of ADMM to DRS. Note that
		this reduction does not need any further assumption 
		on $T$ than $\ran T^*\cap\dom g^*\ne\varnothing$,
		which is weaker than
		$\ker T=\{0\}$, used in \cite[Theorem~3.2]{moursi2019}
		(see also \cite[Appendix A]{russell2016} and 
		\cite[Proposition 3.43]{ecksteinthesis} in finite dimensions). 
		
		\item 
		Theorem~\ref{teo:admmgeneralized} provides the weak 
		convergence of shadow sequences, improving 
		\cite[Theorem 5.1]{gabay83} in the optimization setting.
		In addition, 
		Theorem~\ref{teo:admmgeneralized} 
		recovers the result in \cite[Proposition 3.42]{ecksteinthesis} when 
		$K$ has full column rank in the finite dimensional setting. 
	\end{enumerate}
\end{rem}

The following result allows to deal with more general formulations 
involving two linear operators.

\begin{cor}
	\label{c:admmgen}
	Let $\H$, $\G$, $\J$, and $\K$ be real Hilbert spaces, let 
	$g\in\Gamma_0( \K) $, let $h\in\Gamma_0(\J)$, and let $T: \K 
	\to \G$, $J\colon \J\to\H$, and $K:\G \to \H$ be 
	non-zero bounded linear operators such that $0 \in \sri (\dom g^* - 
	\ran T^*)$, $0 \in \sri (\dom h^* - \ran J^*)$, and 
	$0\in\sri(KT\dom 
	g+J\dom h)$. Consider the 
	convex optimization problem
	\begin{align}
		\label{e:optim2lin}
		\min_{y\in\K}\min_{v\in\J}\quad &g(y)+h(v)\nonumber\\
		\text{s.t.}\quad &KTy+Jv=0,
	\end{align}
	under the assumption that solutions exist.
	In addition, let $\Sigma\colon \G\to\G$ and $\varUpsilon\colon\H\to\H$ 
	be strongly 
	monotone self-adjoint linear operators such that 
	$\Sigma^{-1}-K^*\varUpsilon K$ is monotone, let $p_0 \in \K$, let 
	$v_0\in \J$, let
	$x_0 \in \H$, and consider the routine: 
	\begin{equation}
		\label{eq:admmgeneralized2}
		(\forall n\in\N)\quad 
		\begin{array}{l}
			\left\lfloor
			\begin{array}{l}
				y_{n}=x_n+\varUpsilon (K T p_{n} +Jv_{n})\\[2mm]
				p_{n+1} \in \argm{p \in \K} \big(g (p)+\frac{1}{2}\|T p - 
				(Tp_n- \Sigma K^* y_{n})\|_{\Sigma^{-1}}^2\big)\\
				v_{n+1} \in  
				\argm{v\in \mathcal{J}} \big(h(v)+\frac{1}{2}\|Jv 
				+KTp_{n+1}+\varUpsilon^{-1}x_n\|^2_{\varUpsilon}\big)\\[2mm]
				x_{n+1}=x_n + \varUpsilon  (K T p_{n+1} + 
				Jv_{n+1}).
			\end{array}
			\right.
		\end{array}
	\end{equation}
	Then, there exists $(\hat{y},\hat{v})$ solution to \eqref{e:optim2lin} 
	such that the following hold:
	\begin{enumerate}
		\item 
		\label{c:admmgen1}
		$Tp_n\weakly T\hat{y}$ and $Jv_n\weakly J\hat{v}$.
		\item 
		\label{c:admmgen2}
		Suppose that $\ran T^*=\K$. Then, $p_n\weakly\hat{y}$.
		\item 
		\label{c:admmgen3}
		Suppose that $\ran J^*=\J$. Then, $v_n\weakly \hat{v}$.
	\end{enumerate}
\end{cor}
\begin{proof}
	Note that, by setting $f=(-J)\rhd h\colon q\mapsto 
	\min_{Jv=-q}h(v)$,
	\eqref{e:optim2lin} can be 
	equivalently 
	written as
	\begin{equation}
		\label{e:reductP}
		\min_{y\in\K}\: \left(g(y)+\min_{-Jv=KTy} h(v)\right)\equiv 
		\min_{y\in\K}\: \big(g(y)+f(KTy)\big).
	\end{equation} 
	Since $0\in \sri(\dom 
	h^*-\ran J^*)$, \cite[Corollary~15.28]{1} yields 
	$f=(h^*\circ -J^*)^*\in \Gamma_0(\H)$.
	Hence, the problem in \eqref{e:optim2lin} is a particular 
	instance of Problem~\ref{prob:optipd} and  
	it follows from \eqref{eq:admmgeneralized1}, 
	\eqref{e:Moreau_nonsme}, and an argument analogous to that in 
	\eqref{e:proxcalc} that
	\begin{equation}
		(\forall n\in\N)\quad q_{n+1}
		=\varUpsilon^{-1}\,(\id-  
		\prox^{\varUpsilon^{-1}}_{h^*\circ -J^*})\, 
		(x_n+\varUpsilon KTp_{n+1})
		=-Jv_{n+1},
	\end{equation}
	where $v_{n+1}$ is defined in \eqref{eq:admmgeneralized2}.
	Hence, \eqref{eq:admmgeneralized2} is a particular instance of 
	Algorithm~\ref{algo:43}. Moreover,  \cite[Proposition~12.36(i)]{1} 
	yields 
	$0\in\sri(KT\dom g+J\dom h)=\sri (KT\dom g-\dom f)$
	and 
	Proposition~\ref{prop:relacionproblemasduales}\eqref{prop:relacionproblemasduales1}
	implies the existence of a solution to \eqref{eq:z_opt1}.
	Altogether,
	Theorem~\ref{teo:admmgeneralized}\eqref{teo:admmgeneralizediii}
	asserts that there exists 
	$(\hat{y},\hat{x})\in S_P\times S_D$ 
	such that $(x_{n}, -T p_n,q_n) \weak (\hat{x},-T\hat{y},KT\hat{y})$ 
	and $\hat{u}= - T \hat{y}\in S_{P^*}$. Moreover, since
	$0\in\sri(\dom h^*-\ran J^*)$, it follows from \eqref{e:reductP}
	and \cite[Corollary~15.28(i)]{1} that there exists 
	$\hat{v}\in\J$ such that $(\hat{y},\hat{v})$ is 
	a solution to \eqref{e:optim2lin}.
	In particular, $Tp_n\weakly 
	T\hat{y}$ and
	$q_n=-Jv_n\weakly KT\hat{y}=-J\hat{v}$, which yields 
	\ref{c:admmgen1}. Assertions \ref{c:admmgen2} and 
	\ref{c:admmgen3}
	follow analogously as in the proof of 
	Theorem~\ref{teo:admmgeneralized}\eqref{teo:admmgeneralizediiib}.
\end{proof}

\begin{rem}
	\label{rem:ADMM}
	\begin{enumerate}
		\item
		\label{r:remSADMMShefi} 
		In the context of Corollary~\ref{c:admmgen}, 
		let $U=\varUpsilon^{-1}-K\Sigma K^*$ and 
		$V=\Sigma^{-1}-K^*\varUpsilon K$, which are monotone in view of 
		Proposition~\ref{p:linearop}. Then, Algorithm~\ref{algo:43} can be 
		written equivalently as
		\begin{equation}
			\label{e:BrediesSun}
			\begin{array}{l}
				\left\lfloor
				\begin{array}{l}
					p_{n+1} \in \argm{p \in \K} \Big(g (p)+\frac{1}{2}\|KT p + 
					Jv_n+\varUpsilon^{-1}x_n\|_{\varUpsilon}^2
					+\frac{1}{2}\|p-p_n\|^2_{T^*VT}\Big)\\
					v_{n+1} \in \argm{v\in\J}\Big(h(v)+\frac{1}{2}\|KT p_{n+1} + 
					Jv+\varUpsilon^{-1}x_n\|_{\varUpsilon}^2\Big)\\[2mm]
					x_{n+1}=x_n + \varUpsilon  (K T p_{n+1} +
					Jv_{n+1}),
				\end{array}
				\right.
			\end{array}
		\end{equation}
		which is a non-standard version of the \textit{preconditioned ADMM} 
		(PADMM) \cite{BrediesSun17} without proximal quadratic term in the 
		second optimization problem of \eqref{e:BrediesSun}.
		It considers the \textit{augmented Lagrangian with 
			non-standard 
			metric}
		\begin{equation}
			\label{e:PaugmLag}
			\mathcal{L}_{\varUpsilon}\colon(p,v,x)
			\mapsto 
			g(p)+h(v)+\scal{x}{KTp+Jv}+\frac{1}{2}\|KTp+Jv\|_{\varUpsilon}^2, 
		\end{equation}
		which generalizes the classical augmented Lagrangian 
		$\mathcal{L}_{r\id}$ for some $r>0$. 
		Without the strong 
		monotonicity assumptions used in \cite[Theorem~2.1 \& 
		Theorem~3.1]{BrediesSun17},
		the sequences of algorithm \eqref{eq:admmgeneralized2} are
		well defined and 
		Corollary~\ref{c:admmgen} provides weak convergence.
		Moreover, in the case when 
		$J=-\id$ and $\varUpsilon=r\id$, Corollary~\ref{c:admmgen} 
		ensures convergence under weaker assumptions
		than \cite[Algorithm~2]{shefi} (see also 
		\cite{bot2018} for 
		a variant involving a differentiable convex function). 
		In \cite{Osher11}, a non-standard metric is included only in the 
		multiplier update step of \cite[Algorithm~2]{shefi}, but the convergence 
		of the iterates is not obtained.

		\item In the case when $K=\id$ and $\Sigma = 
		\varUpsilon^{-1}$, 
		the algorithm in \eqref{e:BrediesSun}
		reduces to the ADMM algorithm with the
		augmented Lagrangian with non-standard metric 
		\eqref{e:PaugmLag}, which,  given 
		$(q_0,x_0) 
		\in \H\times \H $, iterates
		\begin{equation}\label{eq:algadmm}
			(\forall n \in \N) \quad 
			\begin{array}{l}
				\left\lfloor
				\begin{array}{l}
					p_{n+1} \in \argm{p \in \K} \big(g (p)+
					\frac{1}{2}\|T p 
					\,+Jv_n+\varUpsilon^{-1}x_{n}\|_{\varUpsilon}^2\big) \\
					v_{n+1} \in \argm{v \in \J} \big(h(v)+
					\frac{1}{2}\|T p_{n+1} 
					\,+Jv+\varUpsilon^{-1}x_{n}\|_{\varUpsilon}^2\big)\\[2mm]
					x_{n+1}=x_n + \varUpsilon(T p_{n+1} +J v_{n+1}).\\
				\end{array}
				\right.
			\end{array}
		\end{equation}
		In the particular case when $\varUpsilon=\tau\id$, it reduces to 
		ADMM \cite{ Boyd2011} and 
		\cite{ecksteinadmm2015,gabay83,gabaymercier,GlowinskyMorrocco75}
		when $J=-\id$.

		\item
		\label{rem:ADMMii}
		As in Remark~\ref{r:remSADMM}\eqref{r:remSADMM1}, 
		sequences 
		$(Tp_n)_{n\in\N}$ and $(Jv_n)_{n\in\N}$ in \eqref{eq:algadmm}
		are unique even if
		the solutions to the optimization problems in 
		\eqref{eq:algadmm} are not unique.
		The uniqueness of $(p_n)_{n\in\N}$ (resp. $(v_n)_{n\in\N}$) is 
		guaranteed, e.g., if $g$ (resp. $h$) is strictly convex or if $\ran 
		T^*=\K$ (resp. $\ran J^*=\J$).
	\end{enumerate}
	
\end{rem}
The following corollary is a direct consequence of 
Theorem~\ref{teo:admmgeneralized} when $T=\id$.

\begin{cor}\label{coro:Splitadmm} In the context of 
	Problem~\ref{prob:optipd}, suppose that $T=\id$ 
	and that there exists a solution to \eqref{eq:z_opt1}. Let 
	$\Sigma\colon \G\to\G$ and $\varUpsilon\colon\H\to\H$ 
	be strongly 
	monotone self-adjoint linear operators such that 
	$\Sigma^{-1}-K^*\varUpsilon K$ is monotone, 
	let $p_0 \in \K  $, let 
	$(q_0,x_0) \in \H\times \H $, and
	consider the sequences $(p_n)_{n\in\N}$ and $(x_n)_{n\in\N}$ 
	generated by the recurrence
	\begin{equation}\label{eq:admmgeneralized_particular}
		(\forall n \in \N)\quad  
		\begin{array}{l}
			\left\lfloor
			\begin{array}{l}
				y_{n}=x_n+\varUpsilon (Kp_n-q_n)\\[1mm]
				p_{n+1} =\prox^{\Sigma^{-1}}_{g}\big(p_n-\Sigma
				K^*y_{n}\big)\\[1mm]
				q_{n+1} = \prox^{\varUpsilon}_{f } (\varUpsilon^{-1}x_n + 
				K p_{n+1} ) \\[1mm]
				x_{n+1}=x_n + \varUpsilon(K  p_{n+1} - q_{n+1}).
			\end{array}
			\right.
		\end{array}
	\end{equation}
	Then, 
	there exists $(\hat{y},\hat{x} )\in S_P\times S_{D}$ such that 
	$(p_n, x_n) \weak (\hat{y},\hat{x})$. 
\end{cor}

\begin{rem}
	\label{rem:corsplit}
	\begin{enumerate}
		\item Note that the explicit method proposed in 
		Corollary~\ref{coro:Splitadmm} includes two multiplier updates as
		the algorithm in \cite[Algorithm~I]{ChenTeb94}. Our method allows for 
		different step-sizes in primal and dual updates and the main 
		distinction 
		is that the third step in \eqref{eq:admmgeneralized_particular}
		includes the information of its second step, while the algorithm in 
		\cite[Algorithm~I]{ChenTeb94} uses the information of previous 
		iteration.
		
		\item Note that \eqref{eq:admmgeneralized_particular} and 
		\eqref{e:Moreau_nonsme}
		yield, for every $n\in\N$,
		\begin{align}
			x_{n+1}&=x_n+\varUpsilon Kp_{n+1}-\varUpsilon 
			q_{n+1}\nonumber\\
			&=\varUpsilon(\id-\prox^{\varUpsilon}_f)
			(\varUpsilon^{-1} x_n+ 
			Kp_{n+1})\nonumber\\
			&=\prox^{\varUpsilon^{-1}}_{f^*}(x_n+\varUpsilon Kp_{n+1})
		\end{align}
		and $y_{n+1}=x_{n+1}+\varUpsilon 
		(Kp_{n+1}-q_{n+1})=2x_{n+1}-x_n$. Therefore, when 
		$\|\varUpsilon^{\frac{1}{2}}\circ K^*\circ \Sigma^{\frac{1}{2}}\| < 1$,
		\eqref{eq:admmgeneralized_particular} reduces to the algorithm 
		proposed in \cite{pockcham} applied to the dual problem $\min 
		(f^*+g^*\circ -K^*)(\H)$. Hence, 
		Corollary~\ref{coro:Splitadmm} is a generalization of 
		\cite[Theorem~1]{pockcham} in this context. 
		
		\item
		\label{rem:corsplit2}
		Observe that the second step in 
		\eqref{eq:admmgeneralized_particular} is explicit, differing from 
		the first step in ADMM \eqref{eq:algadmm}, which is implicit. 
		This feature allows for an algorithm with very low computational 
		cost by iteration. However, 
		the number of iterations may be much larger than those of ADMM
		in some instances, as we verify numerically in 
		Section~\ref{sec:toy}.
	\end{enumerate}
	
\end{rem}

\section{Numerical experiments}
\label{sec:numer}
In this section we provide two numerical experiments. In the first 
experiment we compare SDR with several schemes in the 
literature for 
solving the total variation image restoration problem. In the second 
experiment we consider an 
academic example in which splitting $K$ from $T$
has numerical advantages with respect to ADMM.

\subsection{Total variation image restoration}
\label{sec:TV}
A classical model in image processing is the total variation 
image restoration \cite{ROF}, which aims at 
recovering an image from a blurred and noisy observation under
piecewise constant assumption on the solution.
The model is formulated via the optimization problem
\begin{equation}\label{pro:TV}
	\min_{x \in [0,255]^{N}} 
	\frac{1}{2}\|Rx-b\|^2_{2}+\alpha\|\nabla x\|_1=:F^{TV}(x),
\end{equation} 
where $x\in [0,255]^{N}$ is the image of $N=N_1 \times N_2$ pixels 
to recover from a 
blurred and noisy observation $b\in \R^{m}$, $R : 
\R^{N} \rightarrow \R^{m}$ is a linear operator representing a 
Gaussian blur, the 
discrete gradient
$\nabla\colon x\mapsto \nabla x=(D_1x,D_2x)$ includes horizontal 
and 
vertical differences through linear operators $D_1$ and $D_2$, 
respectively, 
its adjoint $\nabla^*$ is the discrete divergence
(see, e.g., \cite{TV-chambolle}), and $\alpha \in \RPP$. 
A difficulty in this model is 
the presence of the non-smooth $\ell^1$ 
norm composed with the discrete 
gradient operator $\nabla$, which is non-differentiable and 
its proximity operator has not a closed form. 

Note that, by setting $f=\|R\cdot-b\|^2/2$, $g_1=\alpha\|\cdot\|_1$, 
and $g_2=\iota_{[0,255]^N}$, $L_1=\nabla$, 
and $L_2=\id$,
\eqref{pro:TV} can be reformulated as $\min (f+g_1\circ L_1+g_2\circ 
L_2)$ or equivalently as (qualification condition holds)
\begin{equation}
	\text{find}\:\hat{x}\in \R^N\:\:\text{such that}\:\:0\in\partial 
	f(\hat{x})+L_1^*\partial g_1(L_1 x)+L_2^*\partial g_2(L_2 
	\hat{x}),
\end{equation}
which is a particular instance of \eqref{e:vured}, in view of 
\cite[Theorem~20.25]{1}.
Moreover, for every $\T>0$, $J_{\T \partial f}=(\id+\T 
R^*R)^{-1}(\id-\T R^*b)$, for 
every $i\in\{1,2\}$,
$J_{\T (\partial g_i)^{-1}}=\T(\id-\prox_{g_i/\T})(\id/\T)$, 
$\prox_{g_2/\T}=P_{[0,255]^N}$, and 
$\prox_{g_1/\T}=\prox_{\alpha\|\cdot\|_1/\T}$ is the component-wise 
soft thresholder, computed in 
\cite[Example~24.34]{1}. Note that $(\id+\T R^*R)^{-1}$
can be computed efficiently via a diagonalization of 
$R$ using the fast Fourier transform $F$ 
\cite[Section 4.3]{deblur}. 
Altogether, Remark~\ref{rem:SDR}\eqref{rem:SDRvu} allows us to 
write Algorithm~\ref{alg:SDR} as Algorithm~\ref{alg:TV} below, 
where we set $\varUpsilon=\tau\id$, $\Sigma_1=\sigma_1\id$, and 
$\Sigma_2=\sigma_2\id$, for $\T>0$, $\s_1>0$,
and $\s_2>0$. We denote by $\mathcal{R}$  the primal-dual 
error
\begin{equation}
	\label{e:defR}
	\mathcal{R} : (x_+,u_+,x,u) \mapsto 
	\sqrt{\frac{\|(x_+,u_+)-(x,u)\|^2}{\|(x,u)\|^2}}
\end{equation} 
and by $\varepsilon>0$ the convergence tolerance.
The error $\mathcal{R}$ is inspired from \eqref{e:summable} in 
the proof of Theorem~\ref{teo:GDR}.
\begin{algorithm}[H]
	\caption{} 
	\label{alg:TV}
	\begin{algorithmic}[1]
		\STATE{Fix $x_0 \in \R^{N}, v_{1,0}  \in \R^{m}, v_{2,0}  \in 
			\R^{2N} $, 
			$\T\s_1\|\nabla\|^2+\T\s_2\le 
			1$, and $r_0 > \varepsilon>0$. }
		\WHILE{ $r_n > \varepsilon $}
		\STATE{$x_{n+1}= (\id+\T R^*R)^{-1}(x_n-\T \nabla^*v_{1,n} - \T 
			v_{2,n}-\T 
			R^*b)$}
		\STATE{$v_{1,n+1}= 
			\s_1(\id-\prox_{\alpha\|\cdot\|_1/\s_1})(v_{1,n}/\s_1+\nabla(2x_{n+1}-x_n))$}
		\STATE{$v_{2,n+1}=\s_2\big( \id-  P_{[0,255]^N}\big)\left( 
			v_{2,n}/\s_2+ 2x_{n+1}-x_n\right)$}
		\STATE{ 
			$r_n=\mathcal{R}\big((x_{n+1},v_{1,n+1},v_{2,n+1}),
			(x_n,v_{1,n},v_{2,n})\big)$ }
		\ENDWHILE 
		\RETURN{$(x_{n+1},v_{1,n+1},v_{2,n+1})$}
	\end{algorithmic}
\end{algorithm}
In this case, \eqref{e:condvu} reduces to the monotonicity of 
$(\tau^{-1}-\s_2)\id-\s_1 
\nabla^*\nabla$, which is equivalent to 
\begin{equation}
	\label{e:boundary}
	\T\s_1\|\nabla\|^2+\T\s_2\le 1,
\end{equation}
in view of Proposition~\ref{p:linearop}. By using the power iteration 
\cite{vonMises} with tolerance 
$10^{-9}$, we obtain 
$\|\nabla\|^2 \approx 7.9997$. This is consistent with 
\cite[Theorem~3.1]{ChamL}.

Observe that, when $\s_1=\s_2=\s$, Algorithm~\ref{alg:TV} reduces 
to the algorithm proposed in \cite{cp} (when $\s\T(\|\nabla\|^2+1)<1$)
or \cite[Theorem~3.3]{condat}, where the case 
$\s\T(\|\nabla\|^2+1)=1$ is included. 

We provide two main numerical experiments in this subsection: we 
first compare the efficiency of 
Algorithm~\ref{alg:TV} when the step-sizes achieve the boundary in 
\eqref{e:boundary}, verifying that the efficiency is better when the 
equality is achieved. Next, we compare the performance of different 
methods in the literature with optimal step-sizes. For these 
comparisons, we consider the test image of $256\times 256$ pixels 
($N_1=N_2=256$) in 
Figure~\ref{fig:imreal2}\footnote{Image 
	\textit{Circles} obtained 
	from \url{http://links.uwaterloo.ca/Repository.html}} (denoted by 
$\overline{x}$). The 
operator blur $R$ is set as a Gaussian blur of size $9\times 9$ and 
standard 
deviation 4 (applied by MATLAB function 
\textit{fspecial}) and the observation $b$ is obtained by
$b=R\overline{x}+e\in \R^{m_1\times m_2}$, where $m_1=m_2=256$ 
and $e$ is an 
additive zero-mean white Gaussian noise with standard deviation 
$10^{-3}$ (using \textit{imnoise} function in MATLAB). We generate 
20 random realization of random variable $e$ leading to 20
observations $(b_i)_{1\le i\le 20}$. 

In Table~\ref{T:TV_interior_al_borde} we study 
the efficiency of Algorithm~\ref{alg:TV}, in the simpler case when 
$\s_1=\s_2=\s$, as 
parameters $\s$ and $\T$ approach the boundary 
$\s\T(\|\nabla\|^2+1)=1$. In particular, we set 
$\s=\T=\kappa/(10\sqrt{1+\|\nabla\|^2})$ for $\kappa \in \{6,7,8,9,10\}$. 
We provide the averages
of CPU time,  number of iterations, and percentage of error 
between objective values $F^{TV}(\overline{x})$ 
and $F^{TV}(x_n)$ 
obtained by applying Algorithm~\ref{alg:TV} for the 20 observations 
$(b_i)_{1\le i\le 20}$ and for $\kappa\in \{6,7,8,9,10\}$. The tolerance 
is set as $\varepsilon = 10^{-6}$.
We observe that the algorithm becomes more 
efficient (in time and iterations) and accurate (in terms of the 
objective value) as long as 
parameters approach the boundary. This conclusion is confirmed in
Figure~\ref{f:TV_interior_al_borde}, which shows the performance 
obtained 
with the observation $b_{13}$. 
Henceforth, we consider only parameters in the boundary of 
\eqref{e:boundary}.
\begin{table}
	{\footnotesize	 \caption{Averages of
			CPU time, number of iterations, and percentage of error in the 
			objective 
			value obtained from 
			Algorithm~\ref{alg:TV}  with 
			$\T=\s_1=\s_2=\kappa/(10\sqrt{1+\|\nabla\|^2})$ and 
			tolerance $10^{-6}$.}\label{T:TV_interior_al_borde}
		\begin{center}
			\begin{tabular}{|c|c|c|c|}\cline{2-4}
				\multicolumn{1}{c}{} &  \multicolumn{3}{|c|}{$\varepsilon=10^{-6}$} 
				\\ \hline
				\multicolumn{1}{|c|}{$\kappa$}     & Av. Time(s) & Av. Iter. & Av.\% 
				error o.v. \\ \hline
				6     & 43.22 & 8729 & 0.3541 \\\hline
				7     & 40.23 & 8179 & 0.3536 \\\hline
				8     & 38.56 & 7725 & 0.3533 \\\hline
				9     & 36.43 & 7340 & 0.3530 \\\hline
				10    & 34.66 & 7003 & 0.3528 \\
				\hline
			\end{tabular}
	\end{center} }
\end{table}

\begin{figure}
	\centering
	\subfloat[]{\includegraphics[scale=0.2]{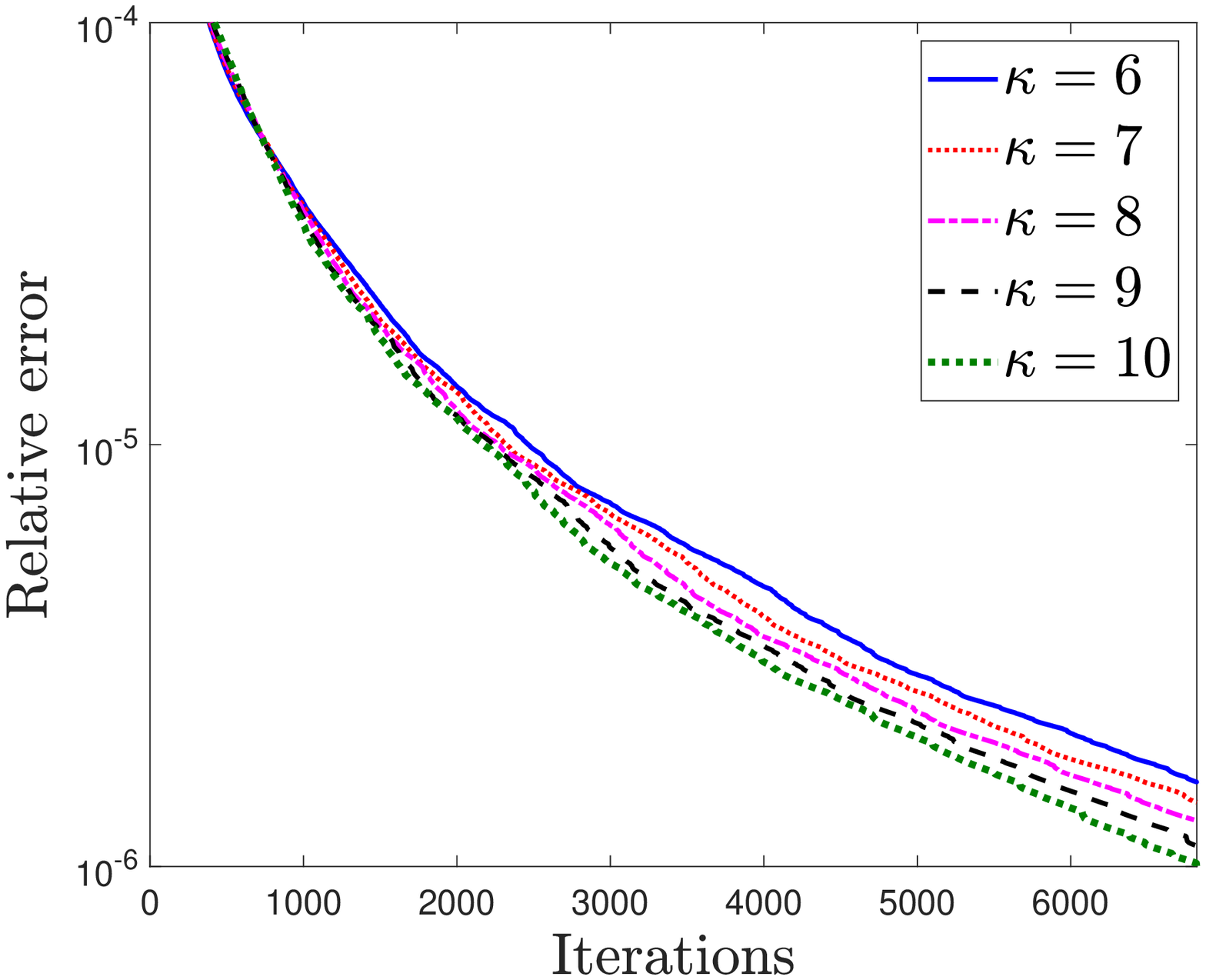}\label{f:graficos_iamges1a}}
	\subfloat[]{\includegraphics[scale=0.2]{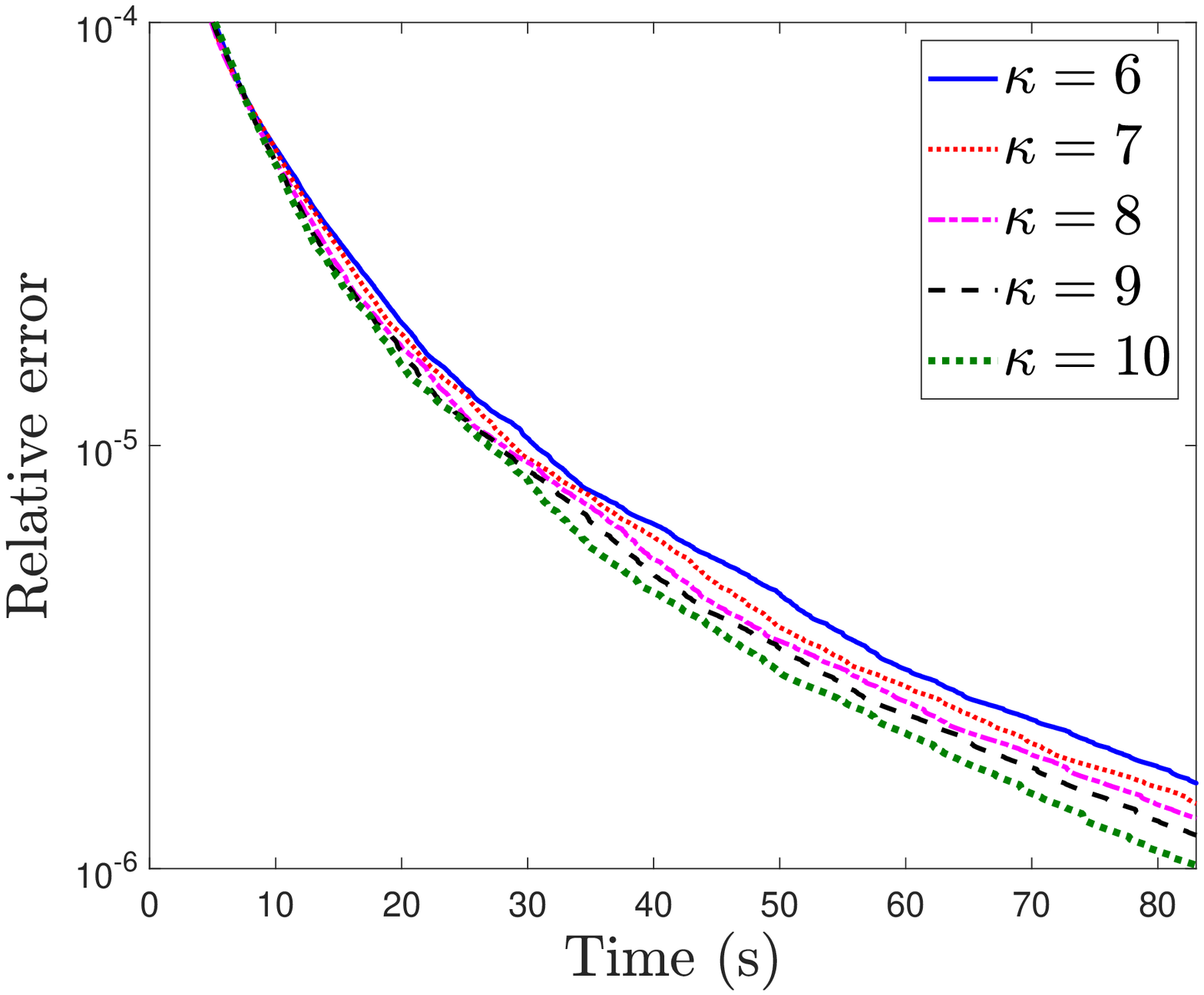}\label{f:graficos_iamges1b}}
	\caption{{Comparison of Algorithm~\ref{alg:TV} with 
			$\T=\s_1=\s_2=\kappa/(10\sqrt{1+\|\nabla\|^2}))$, for 
			image reconstruction from observation $b_{13}$. }}
	\label{f:TV_interior_al_borde}
\end{figure}

Next, we compare 
Algorithm~\ref{alg:TV} when 
$\T\s_1\|\nabla\|^2+\T\s_2= 1$,
with alternative algorithms 
in \cite[Theorem~3.3]{condat},  \cite[Theorem~3.1]{condat} or 
\cite[Corollary~4.2]{vu},
\cite[Theorem~3.1]{skew}, and 
\cite{Molinari2019}, which are called ``Condat'', ``Condat-V\~u'', ``MS'', 
and ``AFBS'', respectively. In order to provide a fair comparison in our 
example, we 
approximate the best step-sizes by considering a mesh on the 
feasible set defined by the conditions allowing convergence for each 
algorithm. In the case when $\varepsilon=10^{-6}$, 
the best performance of Condat-V\~u is obtained by setting 
$\tau=1.2$ and $\s=0.99\cdot (2-\T)/(2\T \|\nabla\|^2)$ which is 
next to the boundary 
of condition $\s\T\|\nabla\|^2<(1-\T/2)$. For MS, the
performance is better when the only step-size $\T$ is next to the 
boundary of the condition $\T<1/\sqrt{1+\|\nabla\|^2}$, which leads us 
to set  
$\T=0.99/\sqrt{1+\|\nabla\|^2}$. For AFBS, we found as best 
parameters 
$\T=0.13$ and $\lambda_n\equiv 
1.7/(65n+10)^{0.505}$ (see \cite{Molinari2019}). In the case of 
Condat, we consider 34 cases of parameters 
$\T$ and $\s$ satisfying $ \s\T(1+\|\nabla\|^2)=1$, by setting
$\T_k=\delta^{k}/(800\sqrt{1+\|\nabla\|^2})$ and  
$\s_k=800/(\delta^{k}\sqrt{1+\|\nabla\|^2})$, where $\delta=800^{1/8}$ 
and $k\in\{1,\ldots,34\}$. For Algorithm~\ref{alg:TV} we consider the 
same parameters $(\T_k)_{1\le k\le 34}$ than those in Condat, and 
we set $\s_{1,k}^{\ell}=(1-\ell)/(\T_k \|\nabla\|^2)$ and 
$\s_{2,k}^{\ell}=\ell/\T_k$, for $\ell \in 10^{-1}\cdot\{5,  0.1,  0.05,  
0.01,  
0.005,  0.003\}$, 
in view of \eqref{e:boundary}.
In Table~\ref{T:tol6} we provide the averages
of CPU time, number of iterations, and the 
percentage of error between objective values $F^{TV}(\overline{x})$ 
and $F^{TV}(x_n)$ obtained by previous algorithms with tolerance 
$\varepsilon=10^{-6}$ considering the observations 
$(b_i)_{1\le i\le 20}$.  We show the best 5 cases for 
Algorithm~\ref{alg:TV} ($k \in 
\{20,\ldots, 24\}$) and the best case for Condat ($k=22$). 
We observe that Algorithm~\ref{alg:TV} and Condat
reduce drastically the computational time and iterations obtained in 
Table~\ref{T:TV_interior_al_borde}, which shows the advantage 
of searching optimal parameters in the boundary of the condition of 
convergence.
We also observe in 
Table~\ref{T:tol6} that Algorithm~\ref{alg:TV} ($k = 22$ and 
$\ell=0.001$) is the 
most efficient method for this tolerance, followed closely by Condat 
($k=22$). Both methods outperform drastically the competitors.
In Figure~\ref{f:TV_comp1} we show the relative error 
versus iterations and time for the observation $b_{13}$, confirming 
previous results.
\begin{table}
	{\footnotesize	 \caption{
			Averages of CPU time, number of iterations, and percentage of error 
			in the 
			objective value for 
			Algorithm~\ref{alg:TV} with $\T\s_1\|\nabla\|^2+\T\s_2= 1$,
			Condat, Condat-V\~u, AFBS, and 
			MS with
			tolerance $10^{-6}$.}\label{T:tol6}
		\begin{center}
			\begin{tabular}{|c|c|c|c|c|c|}\cline{4-6}
				\multicolumn{3}{c}{}  & \multicolumn{3}{|c|}{$\varepsilon=10^{-6}$} 
				\\ \hline
				Algorithm & $\T$     &  $\s_1$ & Av. Time(s) & Av. Iter. & Av. \% error 
				o.v. \\\hline
				\multirow{5}{*}{Alg.\ref{alg:TV} } & 0.77 & 0.16 & 21.12 & 4106 & 
				0.3531 \\
				& 1.17 & 0.11 & 15.33 & 3223 & 0.3562 \\
				& 1.77 & 0.07 & 13.97 & 2787 & 0.3649 \\
				& 2.69 &0.05 & 14.36 & 2891 & 0.3771 \\
				& 4.09 & 0.03 & 16.23 & 3372 & 0.3907 \\\hline
				\multicolumn{1}{|c|}{Condat} &  1.77 & - & 14.89 & 2853 & 0.3673 \\ 
				\hline
				\multicolumn{1}{|c|}{Condat-V\~u} & 1.2 & - & 28.19 & 3539 &  
				0.3738 \\ 
				\hline
				\multicolumn{1}{|c|}{MS} & 0.33 &  - & 62.48 & 6193  & 0.3506 \\ \hline
				\multicolumn{1}{|c|}{AFBS} & 0.13 & - & 85.76 & 11104 & 0.6611\\ 
				\hline
			\end{tabular}
	\end{center} }
\end{table}

\begin{figure}
	\centering
	\subfloat[]{\includegraphics[scale=0.25]{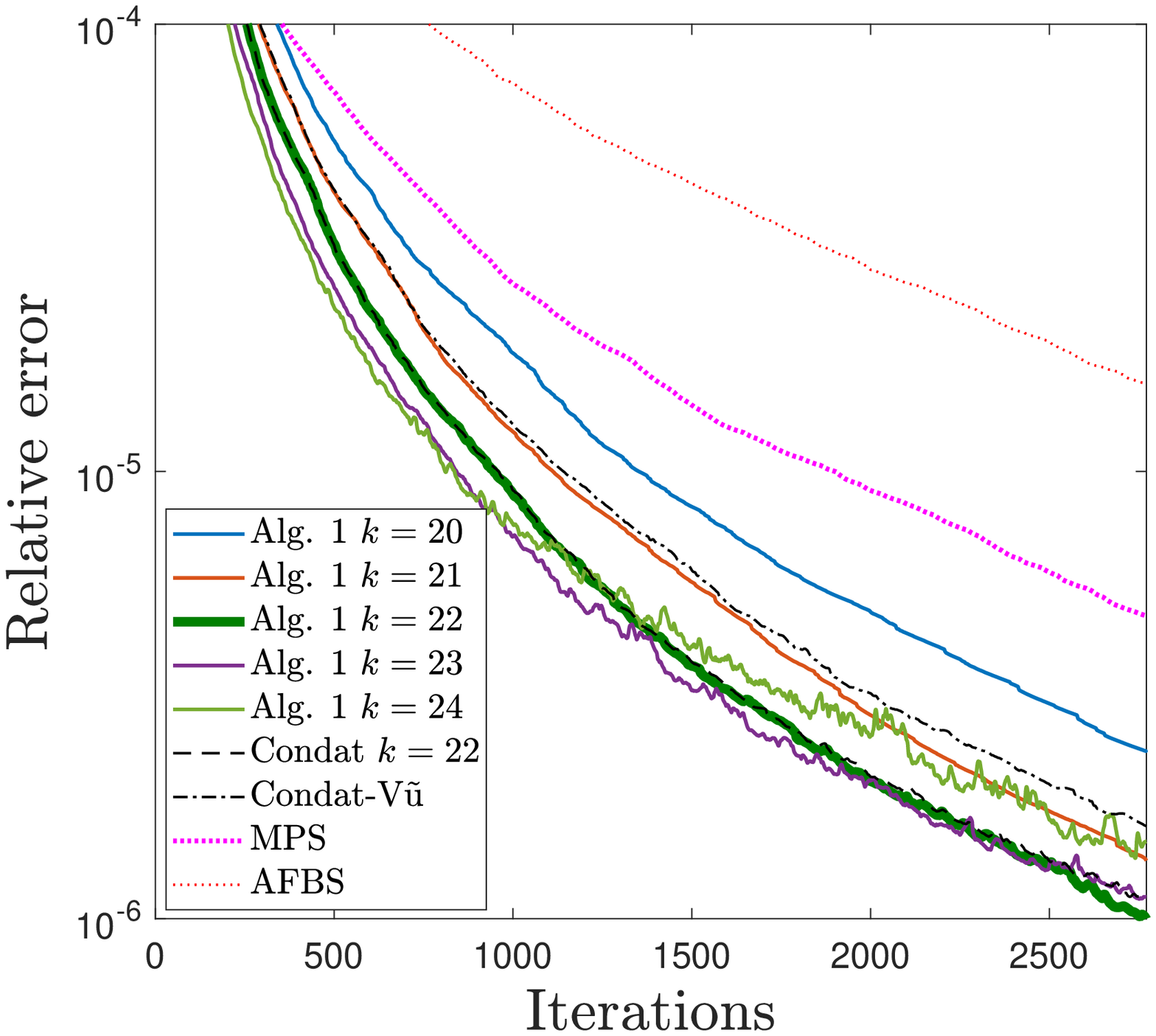}\label{f:comp1a}}
	\subfloat[]{\includegraphics[scale=0.25]{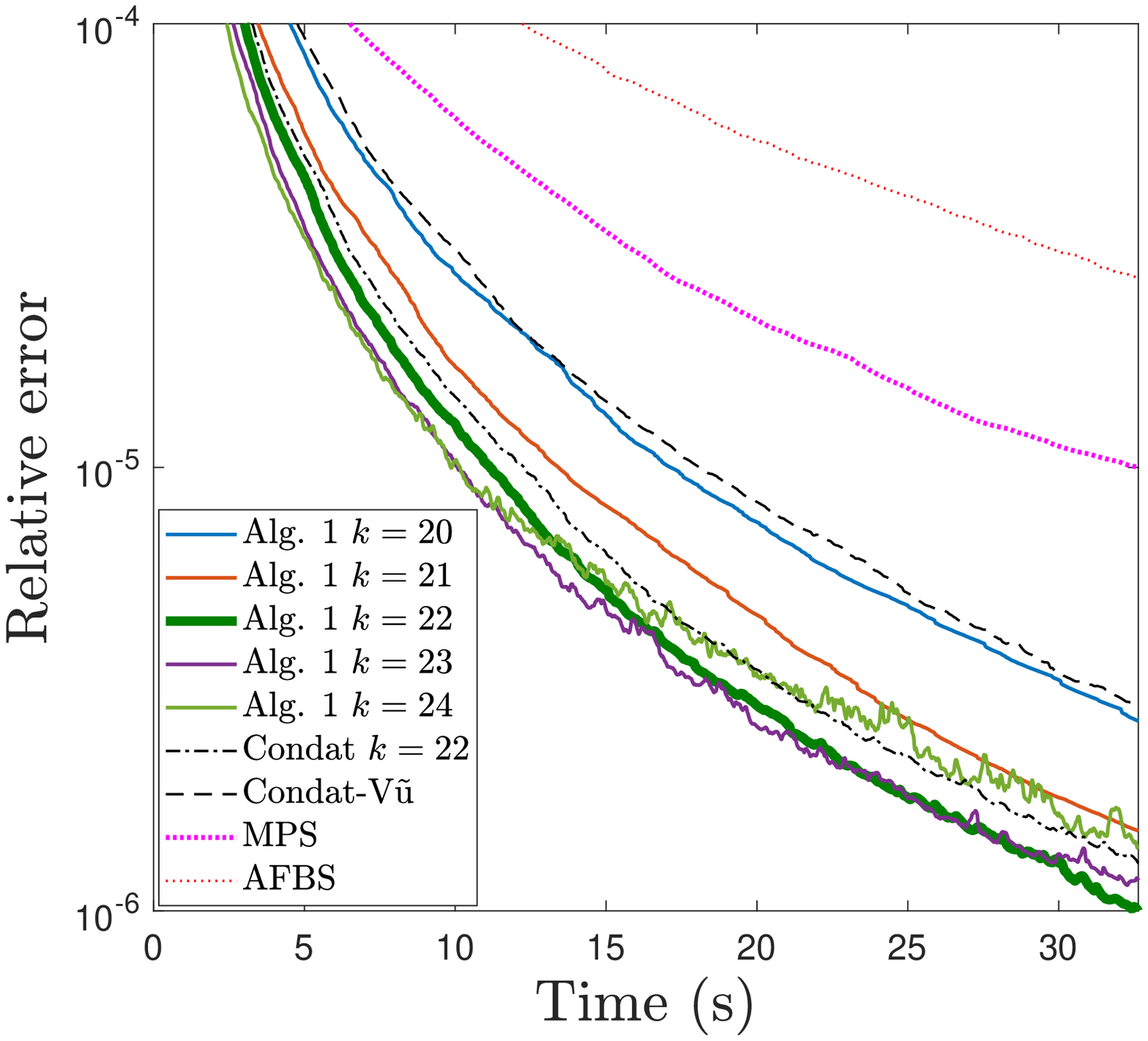}\label{f:comp1b}}
	\caption{{Comparison of Algorithm~\ref{alg:TV} with 
			$\T\s_1\|\nabla\|^2+\T\s_2= 1$, Condat, Condat-V\~u, AFBS, and 
			MS (observation $b_{13}$). }}
	\label{f:TV_comp1}
\end{figure}
In order to make a more precise comparison of Algorithm~\ref{alg:TV} 
and Condat, we consider a smaller tolerance $\varepsilon = 10^{-8}$. 
The obtained results are shown in Table~\ref{T:tol8} and 
Figure~\ref{f:TV_comp2}.
We observe that Algorithm~\ref{alg:TV} ($k = 
21$ and 
$\ell=0.001$) achieves the tolerance in approximately $11\%$ less 
CPU time 
than Condat in its best case ($k=21$). The efficiency in the case of 
the observation $b_{13}$ is illustrated in
Figure~\ref{f:TV_comp2}. 

\begin{table}
	{\footnotesize	 \caption{Averages of CPU time, number of iterations, 
			and percentage of error 
			in the 
			objective value for 
			Algorithm~\ref{alg:TV} with $\T\s_1\|\nabla\|^2+\T\s_2= 1$ and
			Condat with tolerance $10^{-8}$.}\label{T:tol8}
		\begin{center}
			\begin{tabular}{|c|c|c|c|c|c|}\cline{4-6}
				\multicolumn{3}{c}{}  & \multicolumn{3}{|c|}{$\varepsilon=10^{-8}$} 
				\\ \hline
				\multicolumn{1}{|c|}{Algorithm}  & $\T$     &  $\s_1$ & Av. Time(s) & 
				Av. Iter. & Av. \% error o.v. \\\hline
				\multirow{5}{*}{Alg. \ref{alg:TV}} & 0.77 & 0.16 & 93.36 & 19560 & 
				0.3514 \\
				& 1.17 & 0.11 & 83.15 & 17561 & 0.3515 \\
				& 1.77 & 0.07 & 100.06 & 20796 & 0.3515 \\
				& 2.69 & 0.05 & 128.80 & 26801 & 0.3516 \\
				& 4.09 & 0.03 & 160.92 & 33709 & 0.3517 \\\hline 
				\multicolumn{1}{|c|}{Condat} & 1.17 & - & 93.77& 18451 & 0.3515 \\ 
				\hline
			\end{tabular}
	\end{center} }
\end{table}

\begin{figure}
	\centering
	\subfloat[]{\includegraphics[scale=0.25]{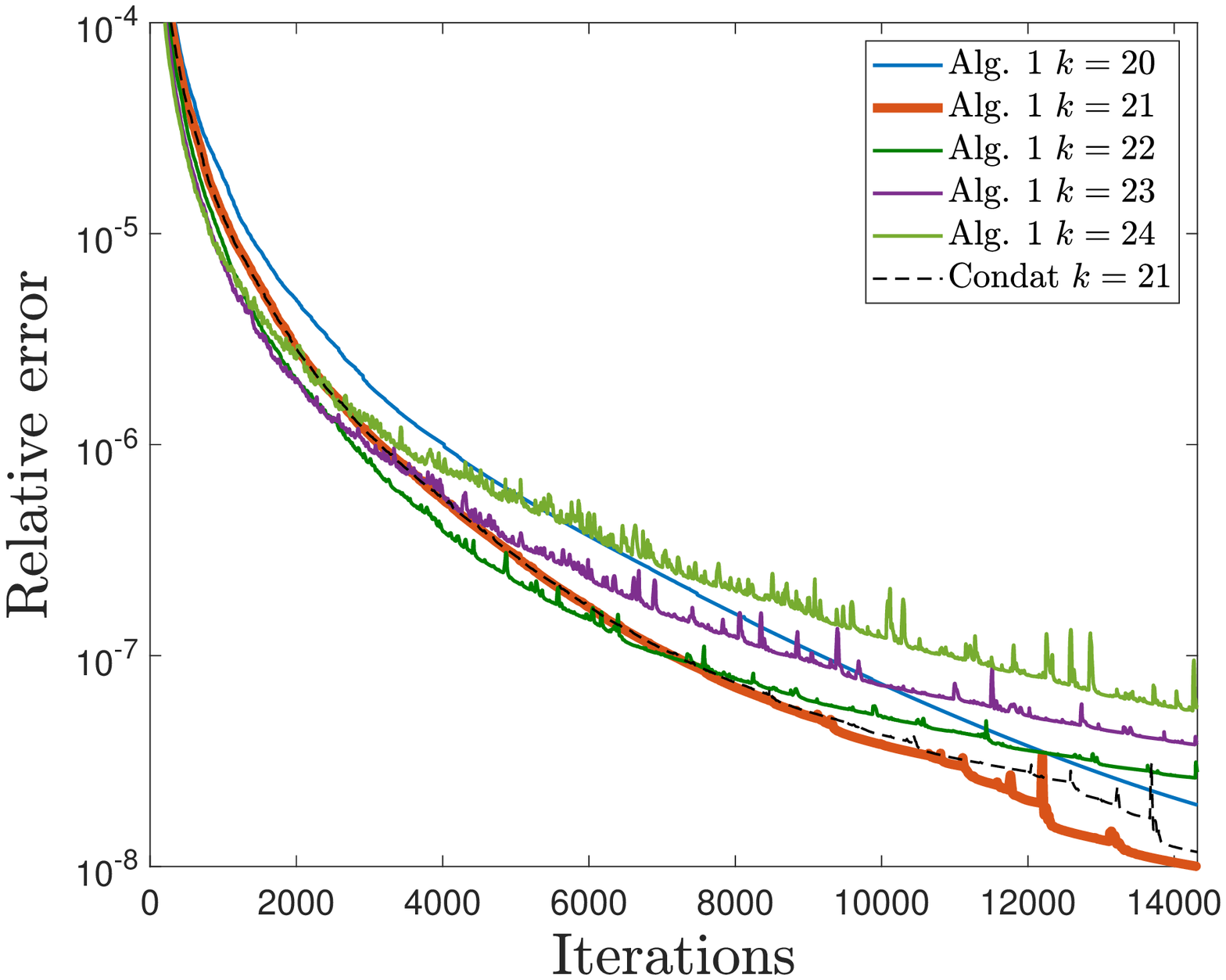}\label{f:comp2a}}
	\subfloat[]{\includegraphics[scale=0.25]{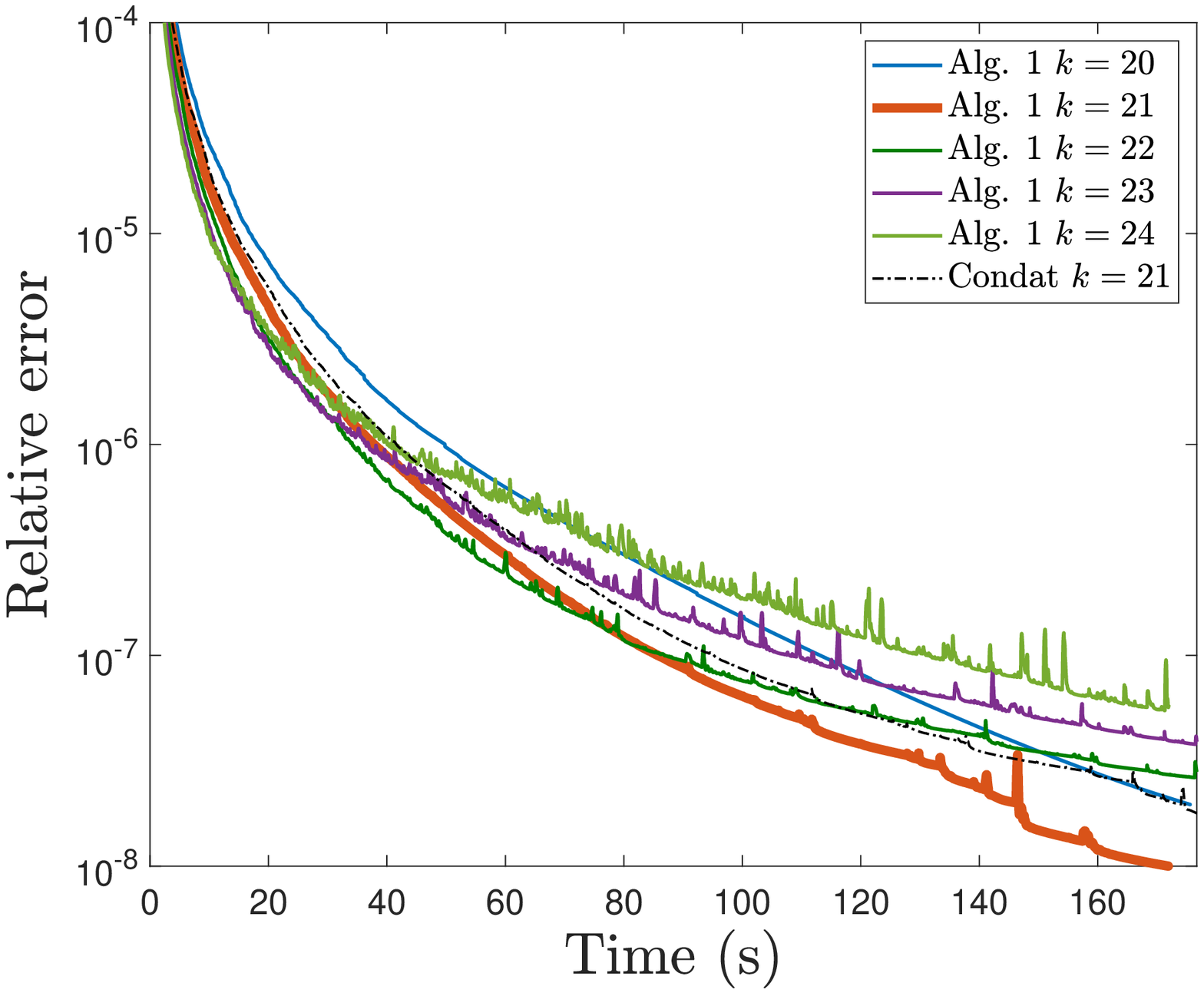}\label{f:comp2b}}
	\caption{{Comparison of Algorithm~\ref{alg:TV} with 
			$\T\s_1\|\nabla\|^2+\T\s_2= 1$ and Condat (observation $b_{13}$). }}
	\label{f:TV_comp2}
\end{figure}
The reconstructed images, after 100 iterations, for the different 
algorithms are shown in Figure~\eqref{f:imagenesTV}. The best 
reconstruction, in terms of objective value $F^{TV}$ and PSNR (Peak 
signal-to-noise ratio), are 
obtained by Condat and Algorithm~\ref{alg:TV}.  

\begin{figure}
	\centering
	\subfloat[\scriptsize Original, $F^{TV}(\overline{x})=10.32$ 
	]{\label{fig:imreal2}\includegraphics[scale=0.3]{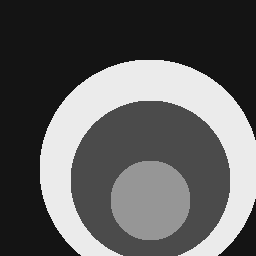}}\,
	\subfloat[\scriptsize Blurry/noisy $b_{13}$, $F^{TV}(b)=53.13$,
	PSNR=$22.15$]{\label{fig:imblur}\includegraphics[scale=0.3]{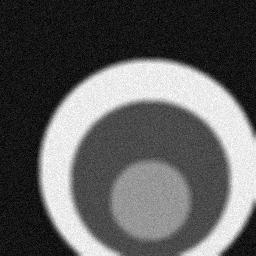}}\,
	\subfloat[\scriptsize AFBS,  
	$F^{TV}(x_{100})=11.28$, 
	PSNR=$25.53$.]{\label{fig:imAFBS}\includegraphics[scale=0.3]{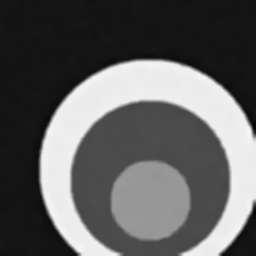}}\,
	\subfloat[\scriptsize MS, $F^{TV}(x_{100})=10.56$, PSNR=$26.38$. 
	]{\label{fig:imMS}\includegraphics[scale=0.3]{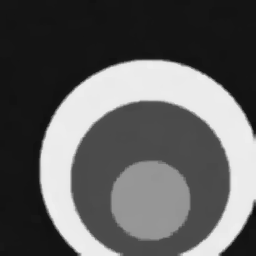}}\\
	\subfloat[\scriptsize Condat-V\~u, $F^{TV}(x_{100})=10.94$, 
	PSNR=$28.38$.]{\label{fig:imFISTATV}\includegraphics[scale=0.3]{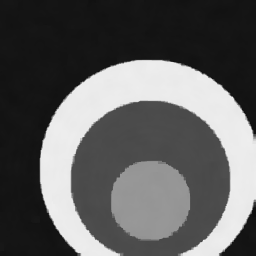}}\,
	\subfloat[\scriptsize Condat ($k=21$), $F^{TV}(x_{100})=10.57$,  
	PSNR=$28.80$.]{\label{fig:imcondat}\includegraphics[scale=0.3]{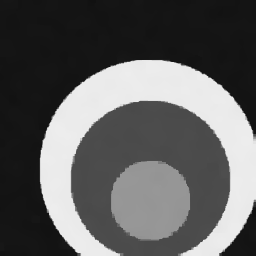}}\,
	\subfloat[\scriptsize Alg.~\ref{alg:TV} ($k=21$, $\ell=0.001$), 
	$F^{TV}(x_{100})=10.55$, 
	PSNR=$28.80$.]{\label{fig:imrealcase14}\includegraphics[scale=0.3]{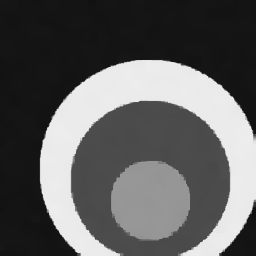}}
	\caption{Reconstructed image, after 100 iterations, from blurred and 
		noisy image using AFBS, MS, Condat-V\~u, Condat  and Alg. 
		\ref{alg:TV}.}
	\label{f:imagenesTV}
\end{figure}

\subsection{Split-ADMM in an academical example}
\label{sec:toy}
In this section, we implement Algorithm~\ref{algo:43}, 
Corollary~\ref{coro:Splitadmm}, and ADMM 
in \eqref{eq:algadmm} for solving an academical example in the 
context of Example~\ref{ex:sadmm}. We compare their 
performances when solving the
following optimization problem
\begin{equation} \label{prob:experiences}
	\min_{x \in \R^N} F(x)= h(x-z)+ \alpha \|Mx\|_1,
\end{equation}
where $h: \R^N \to \R$ is defined by
\begin{equation}\label{eq:huberfunction}
	h\colon x = (\xi_i)_{1\le i\le 
		n}\mapsto \sum_{i=1}^N \phi (\xi_i),\quad \phi : \R \to 
	\R\colon 
	\xi \mapsto \begin{cases}
		|\xi|-\dfrac{\delta}{2}, &\text{ if } |\xi|> \delta;\\
		\dfrac{\xi^2}{2\delta}, &\text{ if } |\xi|\leq \delta, 
	\end{cases}
\end{equation}
$\delta >0$,
$z \in \R^N$, $\alpha > 0$, and $M$  is a $N\times N$ 
symmetric positive definite real matrix. The first term in 
\eqref{prob:experiences}
is a data fidelity penalization using the Huber distance and 
the second 
term imposes sparsity in the solution. This type of 
problems appears 
naturally in image and signal denoising (see, e.g., 
\cite{Pustelnik2019,Liu2014,Pang2017,Sha2020}).

Since $M$ is symmetric, there exist $N\times N$ real 
matrices $P$
and $D$, such that $P^{\top}=P^{-1}$, $D$ is diagonal, and 
$M=PDP^{\top}$.
By setting $g = h(\cdot - z)$,  $f = \alpha \| \cdot \|_1$, 
$K=PD^{1-\eta}P^{\top}$, and $T=PD^{\eta}P^{\top}$, for some 
$\eta\in\left[0,1\right]$, we deduce that 
$KT=M$ and \eqref{prob:experiences} is a particular instance of
\eqref{eq:primalminproblem}. Next, we illustrate the 
efficiency of Algorithm~\ref{algo:43} for different values of 
$\eta\in\left[0,1\right]$.
Observe that, in the case when $\eta=0$ we have $T=\id$ and
Algorithm~\ref{algo:43} reduces to the algorithm in 
Corollary~\ref{coro:Splitadmm}. On the other hand, 
in the case when $\eta=1$ we have $K=\id$ and 
Algorithm~\ref{algo:43} reduces
to ADMM in \eqref{eq:algadmm}. We have
$\prox_{f}\colon (\xi_i)_{1\le i\le n}\mapsto \prox_{|\cdot|}(\xi_i)$, 
where $\prox_{|\cdot|}$ is the scalar soft-thresholder operator  
\cite[Example~24.34(iii)]{1}. 
Note that, since $\ker T=\{0\}$, for every $\eta\in\left[0,1\right]$,
the optimization problem in the second step of 
\eqref{eq:admmgeneralized1} admits a unique solution, in view of 
Remark~\ref{rem:ADMM}\eqref{rem:ADMMii}.
Therefore, when $\varUpsilon=\T\id$ and $\Sigma=\s\id$,
Algorithm~\ref{algo:43} in this example reads as follows.
\begin{algorithm}[H]
	\caption{ }
	\label{alg:SADMM_huber}
	\begin{algorithmic}[1]
		\STATE{Fix $\tau >0$, $p_0, q_0, x_0 \in \R^N$, $\varepsilon>0$, 
			and $r_0 
			> \varepsilon$.}
		\WHILE{ $r_n > \varepsilon $}
		\STATE {
			$y_{n} =\! x_n+\T(KT p_{n}-q_{n})$}
		\STATE\label{alg:SADMM_huber_step3}{$p_{n+1} =\! \zer 
			\big(\s\nabla h 
			(\cdot-z)+T^*\big(T \cdot - (Tp_n-\s K^* 
			y_{n})\big)\big)$} 
		\STATE{$q_{n+1}= \prox_{f/\T } (x_n/\T + KT p_{n+1} )$}
		\STATE{ $x_{n+1}=x_n+\T(KT p_{n+1}-q_{n+1})$ }
		\STATE{ $u_{n+1}= \s K^* (x_{n+1}-x_n)-T p_{n+1}$ }
		\STATE{ $r_{n+1} = \mathcal{R}(x_{n+1},u_{n+1},x_n,u_{n})$ }
		\ENDWHILE 
		\RETURN{$(p_{n+1},q_{n+1},x_{n+1})$}
	\end{algorithmic}
\end{algorithm}

Note that the step \ref{alg:SADMM_huber_step3} in
Algorithm~\ref{alg:SADMM_huber} 
involves 
the resolution of a non-linear equation when $\eta>0$.
On the other hand, in the case when
$\eta=0$, we have $T=\id$ and, as noticed
in Remark~\ref{rem:corsplit}\eqref{rem:corsplit2}, 
the step \ref{alg:SADMM_huber_step3} can be computed 
explicitly by using 
$\prox_{g}=z+\prox_{h}(\cdot-z)$ \cite[Proposition 23.17(iii)]{1}
and the fact that $\delta h$ 
is the real {\it Huber function} (see \cite[Example 
8.44\,\&\,Example 24.9]{1}).
We consider as stopping criterion the primal-dual relative error
defined in \eqref{e:defR}. 

We compare the performance of Algorithm~\ref{alg:SADMM_huber} 
when $\eta \in \{0, 0.8, 0.9, 1\}$
with the standard solver \textit{fmincon} of 
MATLAB for $N \in \{ 100, 250, 
500\}$ and different values of the 
minimum and maximum
eigenvalues $\lambda_{\max}\ge\lambda_{\min}>0$ 
of the matrix $M$.
Since the expected value of 
$\lambda_{\max}$ (resp. $\lambda_{\min}$) of random matrices
generated by a normal distribution increases (resp. decreases) as 
$N$ increases (see 
\cite[Table~1.2]{Edelman1988}), we consider three classes 
of matrices with condition number 
$\kappa=\lambda_{\max}/\lambda_{\min}=50$
for each dimension $N \in \{ 100, 250, 500\}$:
\begin{itemize}
	\item \textbf{Class A:} 
	Class of matrices $M$ with small eigenvalues 
	($\lambda_{\max}=N/1000$).
	
	\item \textbf{Class B:} Class of matrices $M$ with average 
	eigenvalues ($\lambda_{\max}=4N$). 
	
	\item \textbf{Class C:} Class of matrices $M$ with large 
	eigenvalues ($\lambda_{\max}=100N$). 
\end{itemize}
For each class, we generate 30 random matrices using the {\it randn} 
function of MATLAB and the eigenvalues of each randomly 
generated matrix $M$ is forced to satisfy the conditions of each class 
after a singular value decomposition $M=PDP^{\top}$. 
We next generate $T$ and $K$ as described before. 
Step~\ref{alg:SADMM_huber_step3} in 
Algorithm~\ref{alg:SADMM_huber} is computed via {\it fsolve} 
function of MATLAB (for $\eta >0$). We define the \textit{percentage 
	of improvement} of an algorithm with respect to \textit{fmincon} via
$I_{\bar{n}}=(\overline{F}-F(p_{\bar{n}}))\cdot
100/\overline{F},$
where $\overline{F}$ stands for the value of the function obtained 
by \textit{fmincon} with tolerance $10^{-14}$ and $F(p_{\bar{n}})$ is 
the value of the function
obtained by Algorithm~\ref{alg:SADMM_huber} when it stops in 
iteration $\bar{n}$.
Finally, we set the tolerance $\varepsilon = 10^{-6}$ 
and $\T =1$ in Algorithm~\ref{alg:SADMM_huber}. 

Table~\ref{T:comparacion_algoritmos_n=100} provides 
the averages of CPU time, iterations, and percentage of 
improvement with respect to 
\textit{fmincon} of Algorithm~\ref{alg:SADMM_huber} in the cases 
$\eta \in \{0, 0.8, 0.9, 1\}$ for the 30 random matrices in each 
class and $N\in\{100,250,500\}$. 
We split our 
analysis of the results in the three classes of random matrices.

The best performance in the class A (small eigenvalues) is
obtained by the case when $\eta=0$ 
(Corollary~\ref{coro:Splitadmm}) in 
each dimension. The function value is very close to the one 
obtained by 
\textit{fmincon} 
(difference of $10^{-5}$\%). For this class, the cases
when $\eta\in\{0.8,0.9\}$ are less accurate and ADMM ($\eta=1$)
is even more precise but much slower than the case when 
$\eta=0$ for 
this class. This is explained by a very low cost per iteration 
and a comparable average number of iterations of the 
case when $\eta=0$. 

On the other hand, for matrices belonging to the class B (average 
eigenvalues), the most efficient method is SADMM with 
$\eta=0.9$. The method needs very few number of iterations on 
average and it is more accurate than \textit{fmincon}, since 
$I_{\bar{n}}$ is positive. This feature is also verified in 
$\eta\in\{0.8,1\}$ but the number of iterations and computational time
is larger. We observe that the case when $\eta=0$ shows a very 
large 
number of iterations for achieving convergence and looses 
precision as the dimension increases. We conclude that SADMM 
outperforms drastically ADMM and the algorithm of 
Corollary~\ref{coro:Splitadmm}, for suitable factorizations of 
matrices $M$ with average eigenvalues. 

Finally, ADMM ($\eta=1$) is the best algorithm for the class C.
It needs a very few number of iterations on average for achieving 
convergence, which nicely scales with the dimension. The 
computational time is around $1/3$ of the closest competitor and the 
precision is as good as \textit{fmincon}. SADMM algorithms 
when $\eta\in\{0.8,0.9\}$ are similarly accurate but much slower. 
The case when $\eta=0$ is very far from the solution and 
extremely slow
for this class in all dimensions. 

\begin{table}[h!]
	{\footnotesize
		\caption{Performance of Algorithm~\ref{alg:SADMM_huber} 
			for
			$N\in \{100,250, 500\}$,
			$\eta\in\{0,0.8,0.9,1\}$ and classes A, B, and C.}  
		\label{T:comparacion_algoritmos_n=100}
		\begin{center}
			\begin{tabular}{|c|c|c|c|c|c|c|}\hline 
				$N$ & Class & $\eta$ & $0$ & $0.8$ &  
				$0.9$ & $1$ \\ \hline \hline
				\multirow{9}{*}{100}&
				\multirow{3}{*}{A} & Av. time & 0.019 & 4.86 & 
				4.92 & 4.37  \\ \cline{3-7}
				& & Av. iter & 688 & 704 & 717 & 656  \\ \cline{3-7}
				& & Av. $I_{\bar{n}}$ (\%) & -$1.8 \cdot 10^{-5}  $   & -0.47 & -0.07 & 
				-$1.5  \cdot 
				10^{-6}  $ \\ \cline{2-7} \cline{2-7}
				& \multirow{3}{*}{B} & Av. time  & 17.52 & 1.15 & 0.50 & 
				5.41  \\ \cline{3-7}
				& & Av. iter & 798258 & 118 & 49 & 519 \\ \cline{3-7}
				& & Av. $I_{\bar{n}}$ (\%) & 0.63& 0.36 & 0.33 & 0.64 \\ \cline{2-7} \cline{2-7}
				& \multirow{3}{*}{C} & Av. time & 31.44 & 3.77 & 1.07 & 
				0.34   \\ \cline{3-7}
				& & Av. iter & 1410638 & 395 & 107 & 30  \\ \cline{3-7}
				& & Av. $I_{\bar{n}}$ (\%) & -1607  & -$8.4  \cdot 10^{-8}  $ & -$8.1  
				\cdot 10^{-8}  
				$  & -$5.1  \cdot 10^{-8}  $   \\ \hline \hline 
				\multirow{9}{*}{250} & \multirow{3}{*}{A} & Av. time & 0.036 & 8.94 & 
				9.25 & 8.88 \\ \cline{3-7}
				& & Av. iter & 380 & 359 & 387 & 393  \\ \cline{3-7}
				& & Av. $I_{\bar{n}}$ (\%) & -$1.6\cdot 10^{-5}$ & -1.03 & -0.18 &  
				-$8\cdot 10^{-6}  $ \\ \cline{2-7} \cline{2-7}
				& \multirow{3}{*}{B} & Av. time & 136.82 & 5.54 & 2.61 &  
				32.15 \\ \cline{3-7}
				& & Av. iter & 1547593 & 143 & 64  & 886  \\ \cline{3-7}
				& & Av. $I_{\bar{n}}$ (\%) & -0.15 & 0.18 & 0.19 & 0.25 \\ \cline{2-7} \cline{2-7}
				& \multirow{3}{*}{C} & Av. time & 85.28 & 27.14 &5.83 &  
				1.76 \\ \cline{3-7}
				& & Av. iter & 971230  & 761 &  120 & 39 \\ \cline{3-7}
				& & Av. $I_{\bar{n}}$ (\%) & -18287  & -$ 1.3\cdot 10^{-7}$ &  -$9.5 
				\cdot 10^{-8}$  
				&  -$3.3\cdot 10^{-8}$ \\ \hline \hline
				\multirow{9}{*}{500}& \multirow{3}{*}{A} & Av. time & 0.067 & 13.41  &  
				13.58  &  13.52    \\ \cline{3-7}
				& & Av. iter & 123 & 128 & 129 & 132  \\ \cline{3-7}
				& & Av. $I_{\bar{n}}$ (\%) & $7.2 \cdot 10^{-5}$ & -1.47 & -0.30& $8.2 
				\cdot 
				10^{-5}$ \\ \cline{2-7} \cline{2-7}
				& \multirow{3}{*}{B} & Av. time & 581.25 & 39.99 & 23.95 
				& 113.24  \\ \cline{3-7}
				& & Av. iter & 1249041 & 248 & 162 & 740  \\ \cline{3-7}
				& & Av. $I_{\bar{n}}$ (\%) & -2.32  & 0.13 & 0.13 & 0.15 \\ \cline{2-7} \cline{2-7}
				& \multirow{3}{*}{C} & Av. time & 205.34 & 193.95 & 32.09 
				& 12.31  \\ \cline{3-7}
				& & Av. iter & 419896 & 1200 & 182 & 46  \\ \cline{3-7}
				& & Av. $I_{\bar{n}}$(\%) & -261808  & -$1.8\cdot 10^{-7}$ & -$1.5\cdot 
				10^{-7}$ 
				& -$9.4 \cdot 10^{-8}$ \\ \hline 
			\end{tabular}
	\end{center}}
\end{table}

\section*{Acknowledgments}
{\small The first author thanks the support of ANID under grant 
	FON\-DECYT 1190871 
	and grant Redes 180032. The second author thanks the support 
	of ANID-Subdirección de Capital  Humano/Doctorado 
	Nacional/2018-21181024 and of the Direcci\'on
	de Postgrado y Programas from UTFSM through Programa de 
	Incentivos a la Iniciaci\'on
	Cient\'ifica (PIIC).}


\begin{thebibliography}{10}
	
	\bibitem{russell2016}
	{\sc T.~Aspelmeier, C.~Charitha, and D.~R. Luke}, {\em Local linear convergence
		of the {ADMM}/{D}ouglas-{R}achford algorithms without strong convexity and
		application to statistical imaging}, SIAM J. Imaging Sci., 9 (2016),
	pp.~842--868.
	
	\bibitem{Aubin1990}
	{\sc J.-P. Aubin and H.~Frankowska}, {\em Set-valued analysis}, vol.~2 of
	Systems \& Control: Foundations \& Applications, Birkh\"{a}user Boston, Inc.,
	Boston, MA, 1990.
	
	\bibitem{1}
	{\sc H.~H. Bauschke and P.~L. Combettes}, {\em Convex Analysis and Monotone
		Operator Theory in {H}ilbert Spaces}, CMS Books in Mathematics/Ouvrages de
	Math\'{e}matiques de la SMC, Springer, Cham, second~ed., 2017.
	
	\bibitem{bot2018}
	{\sc R.~I. Bo\c{t} and E.~R. Csetnek}, {\em A{DMM} for monotone operators:
		convergence analysis and rates}, Adv. Comput. Math., 45 (2019), pp.~327--359.
	
	\bibitem{BOT2}
	{\sc R.~I. Bo\c{t}, E.~R. Csetnek, and A.~Heinrich}, {\em A primal-dual
		splitting algorithm for finding zeros of sums of maximal monotone operators},
	SIAM J. Optim., 23 (2013), pp.~2011--2036.
	
	\bibitem{bot2013}
	{\sc R.~I. Bo\c{t} and C.~Hendrich}, {\em A {D}ouglas-{R}achford type
		primal-dual method for solving inclusions with mixtures of composite and
		parallel-sum type monotone operators}, SIAM J. Optim., 23 (2013),
	pp.~2541--2565.
	
	\bibitem{Boyd2011}
	{\sc S.~Boyd, N.~Parikh, E.~Chu, B.~Peleato, and J.~Eckstein}, {\em Distributed
		optimization and statistical learning via the alternating direction method of
		multipliers}, Foundations and Trends in Machine Learning, 3 (2011),
	pp.~1--122.
	
	\bibitem{BrediesSun15}
	{\sc K.~Bredies and H.~Sun}, {\em Preconditioned {D}ouglas-{R}achford splitting
		methods for convex-concave saddle-point problems}, SIAM J. Numer. Anal., 53
	(2015), pp.~421--444.
	
	\bibitem{BrediesSun17}
	{\sc K.~Bredies and H.~Sun}, {\em A proximal point analysis of the
		preconditioned alternating direction method of multipliers}, J. Optim. Theory
	Appl., 173 (2017), pp.~878--907.
	
	\bibitem{BrediesSun15TV}
	{\sc K.~Bredies and H.~P. Sun}, {\em Preconditioned {D}ouglas-{R}achford
		algorithms for {TV}- and {TGV}-regularized variational imaging problems}, J.
	Math. Imaging Vision, 52 (2015), pp.~317--344.
	
	\bibitem{Nets1}
	{\sc L.~Brice\~{n}o, R.~Cominetti, C.~E. Cort\'{e}s, and F.~Mart\'{\i}nez},
	{\em An integrated behavioral model of land use and transport system: a
		hyper-network equilibrium approach}, Netw. Spat. Econ., 8 (2008),
	pp.~201--224.
	
	\bibitem{skew}
	{\sc L.~M. {Brice\~{n}o-Arias} and P.~L. Combettes}, {\em A monotone + skew
		splitting model for composite monotone inclusions in duality}, SIAM J.
	Optim., 21 (2011), pp.~1230--1250.
	
	\bibitem{Nash}
	{\sc L.~M. Brice\~{n}o Arias and P.~L. Combettes}, {\em Monotone operator
		methods for {N}ash equilibria in non-potential games}, in Computational and
	analytical mathematics, vol.~50 of Springer Proc. Math. Stat., Springer, New
	York, 2013, pp.~143--159.
	
	\bibitem{Siopt2}
	{\sc L.~M. {Brice\~{n}o-Arias} and D.~Davis}, {\em Forward-backward-half
		forward algorithm for solving monotone inclusions}, SIAM J. Optim., 28
	(2018), pp.~2839--2871.
	
	\bibitem{BAR19}
	{\sc L.~M. {Brice\~no-Arias} and F.~{Rold\'an}}, {\em Primal-dual splittings as
		fixed point iterations in the range of linear operators}, 2019,
	\url{https://arxiv.org/abs/1910.02329}.
	
	\bibitem{ChamL}
	{\sc A.~Chambolle}, {\em An algorithm for total variation minimization and
		applications}, J. Math. Imaging Vision, 20 (2004), pp.~89--97,
	\url{https://doi.org/10.1023/B:JMIV.0000011320.81911.38}.
	
	\bibitem{TV-chambolle}
	{\sc A.~Chambolle, V.~Caselles, D.~Cremers, M.~Novaga, and T.~Pock}, {\em An
		introduction to total variation for image analysis}, in Theoretical
	Foundations and Numerical Methods for Sparse Recovery, vol.~9 of Radon Ser.
	Comput. Appl. Math., Walter de Gruyter, Berlin, 2010, pp.~263--340.
	
	\bibitem{ChambLions97}
	{\sc A.~Chambolle and P.-L. Lions}, {\em Image recovery via total variation
		minimization and related problems}, Numer. Math., 76 (1997), pp.~167--188.
	
	\bibitem{cp}
	{\sc A.~Chambolle and T.~Pock}, {\em A first-order primal-dual algorithm for
		convex problems with applications to imaging}, J. Math. Imaging Vision, 40
	(2011), pp.~120--145.
	
	\bibitem{ChenTeb94}
	{\sc G.~Chen and M.~Teboulle}, {\em A proximal-based decomposition method for
		convex minimization problems}, Math. Programming, 64 (1994), pp.~81--101.
	
	\bibitem{Pustelnik2019}
	{\sc J.~Colas, N.~Pustelnik, C.~Oliver, P.~Abry, J.-C. G{\'e}minard, and
		V.~Vidal}, {\em {Nonlinear denoising for characterization of solid friction
			under low confinement pressure}}, {Physical Review E }, 42 (2019), p.~91.
	
	\bibitem{ipa}
	{\sc P.~L. Combettes}, {\em Quasi-{F}ej\'{e}rian analysis of some optimization
		algorithms}, in Inherently Parallel Algorithms in Feasibility and
	Optimization and their Applications ({H}aifa, 2000), vol.~8 of Stud. Comput.
	Math., North-Holland, Amsterdam, 2001, pp.~115--152.
	
	\bibitem{CombVu14}
	{\sc P.~L. Combettes and B.~C. V\~{u}}, {\em Variable metric forward-backward
		splitting with applications to monotone inclusions in duality}, Optimization,
	63 (2014), pp.~1289--1318.
	
	\bibitem{condat}
	{\sc L.~Condat}, {\em A primal-dual splitting method for convex optimization
		involving {L}ipschitzian, proximable and linear composite terms}, J. Optim.
	Theory Appl., 158 (2013), pp.~460--479.
	
	\bibitem{Vu15}
	{\sc D.~D\~{u}ng and B.~C. V\~{u}}, {\em A splitting algorithm for system of
		composite monotone inclusions}, Vietnam J. Math., 43 (2015), pp.~323--341.
	
	\bibitem{Daube04}
	{\sc I.~Daubechies, M.~Defrise, and C.~De~Mol}, {\em An iterative thresholding
		algorithm for linear inverse problems with a sparsity constraint}, Comm. Pure
	Appl. Math., 57 (2004), pp.~1413--1457.
	
	\bibitem{DR1956}
	{\sc J.~Douglas, Jr. and H.~H. Rachford, Jr.}, {\em On the numerical solution
		of heat conduction problems in two and three space variables}, Trans. Amer.
	Math. Soc., 82 (1956), pp.~421--439.
	
	\bibitem{ecksteinthesis}
	{\sc J.~Eckstein}, {\em Splitting Methods for Monotone Operators with
		Applications to Parallel Optimization}, PhD thesis, Massachusetts Institute
	of Technology, 1989.
	
	\bibitem{eckstein-bertsekas}
	{\sc J.~Eckstein and D.~P. Bertsekas}, {\em On the {D}ouglas-{R}achford
		splitting method and the proximal point algorithm for maximal monotone
		operators}, Math. Programming, 55 (1992), pp.~293--318.
	
	\bibitem{EckSvaiter}
	{\sc J.~Eckstein and B.~F. Svaiter}, {\em A family of projective splitting
		methods for the sum of two maximal monotone operators}, Math. Program., 111
	(2008), pp.~173--199.
	
	\bibitem{ecksteinadmm2015}
	{\sc J.~Eckstein and W.~Yao}, {\em Understanding the convergence of the
		alternating direction method of multipliers: theoretical and computational
		perspectives}, Pac. J. Optim., 11 (2015), pp.~619--644.
	
	\bibitem{Edelman1988}
	{\sc A.~Edelman}, {\em Eigenvalues and condition numbers of random matrices},
	SIAM J. Matrix Anal. Appl., 9 (1988), pp.~543--560.
	
	\bibitem{Fuku96}
	{\sc M.~Fukushima}, {\em The primal {D}ouglas-{R}achford splitting algorithm
		for a class of monotone mappings with application to the traffic equilibrium
		problem}, Math. Programming, 72 (1996), pp.~1--15.
	
	\bibitem{gabay83}
	{\sc D.~Gabay}, {\em Chapter {IX} applications of the method of multipliers to
		variational inequalities}, in Augmented Lagrangian Methods: Applications to
	the Numerical Solution of Boundary-Value Problems, M.~Fortin and
	R.~Glowinski, eds., vol.~15 of Studies in Mathematics and Its Applications,
	Elsevier, 1983, pp.~299 -- 331.
	
	\bibitem{gabaymercier}
	{\sc D.~Gabay and B.~Mercier}, {\em A dual algorithm for the solution of
		nonlinear variational problems via finite element approximation}, Computers
	\& Mathematics with Applications, 2 (1976), pp.~17--40.
	
	\bibitem{GafniBert84}
	{\sc E.~M. Gafni and D.~P. Bertsekas}, {\em Two-metric projection methods for
		constrained optimization}, SIAM J. Control Optim., 22 (1984), pp.~936--964.
	
	\bibitem{GlowinskyMorrocco75}
	{\sc R.~Glowinski and A.~Marrocco}, {\em Sur l'approximation, par
		\'{e}l\'{e}ments finis d'ordre un, et la r\'{e}solution, par
		p\'{e}nalisation-dualit\'{e}, d'une classe de probl\`emes de {D}irichlet non
		lin\'{e}aires}, Rev. Fran\c{c}aise Automat. Informat. Recherche
	Op\'{e}rationnelle S\'{e}r. Rouge Anal. Num\'{e}r., 9 (1975), pp.~41--76.
	
	\bibitem{Goldstein64}
	{\sc A.~A. Goldstein}, {\em Convex programming in {H}ilbert space}, Bull. Amer.
	Math. Soc., 70 (1964), pp.~709--710.
	
	\bibitem{deblur}
	{\sc P.~C. Hansen, J.~G. Nagy, and D.~P. O'Leary}, {\em Deblurring images:
		Matrices, spectra, and filtering}, vol.~3 of Fundamentals of Algorithms,
	Society for Industrial and Applied Mathematics (SIAM), Philadelphia, PA,
	2006.
	
	\bibitem{yuan}
	{\sc B.~He and X.~Yuan}, {\em Convergence analysis of primal-dual algorithms
		for a saddle-point problem: from contraction perspective}, SIAM J. Imaging
	Sci., 5 (2012), pp.~119--149.
	
	\bibitem{lions-mercier79}
	{\sc P.-L. Lions and B.~Mercier}, {\em Splitting algorithms for the sum of two
		nonlinear operators}, SIAM J. Numer. Anal., 16 (1979), pp.~964--979.
	
	\bibitem{Liu2014}
	{\sc X.~Liu, D.~Zhai, D.~Zhao, G.~Zhai, and W.~Gao}, {\em Progressive image
		denoising through hybrid graph {L}aplacian regularization: a unified
		framework}, IEEE Trans. Image Process., 23 (2014), pp.~1491--1503.
	
	\bibitem{vonMises}
	{\sc R.~V. Mises and H.~Pollaczek-Geiringer}, {\em Praktische verfahren der
		gleichungsauflösung}, Zeitschrift für Angewandte Mathematik und Mechanik, 9
	(1929), pp.~152--164.
	
	\bibitem{Molinari2019}
	{\sc C.~Molinari, J.~Peypouquet, and F.~Roldan}, {\em Alternating
		forward-backward splitting for linearly constrained optimization problems},
	Optim. Lett., 14 (2020), pp.~1071--1088.
	
	\bibitem{moursi2019}
	{\sc W.~M. Moursi and Y.~Zinchenko}, {\em A note on the equivalence of operator
		splitting methods}, in Splitting Algorithms, Modern Operator Theory, and
	Applications, Springer, Cham, 2019, pp.~331--349.
	
	\bibitem{Pang2017}
	{\sc J.~Pang and G.~Cheung}, {\em Graph {L}aplacian regularization for image
		denoising: analysis in the continuous domain}, IEEE Trans. Image Process., 26
	(2017), pp.~1770--1785.
	
	\bibitem{pockcham}
	{\sc T.~Pock and A.~Chambolle}, {\em Diagonal preconditioning for first order
		primal-dual algorithms in convex optimization}, in 2011 International
	Conference on Computer Vision, 2011, pp.~1762--1769.
	
	\bibitem{RieszNagy}
	{\sc F.~Riesz and B.~Sz.-Nagy}, {\em Functional analysis}, Frederick Ungar
	Publishing Co., New York, 1955.
	
	\bibitem{ROF}
	{\sc L.~I. Rudin, S.~Osher, and E.~Fatemi}, {\em Nonlinear total variation
		based noise removal algorithms}, Phys. D, 60 (1992), pp.~259--268.
	
	\bibitem{Sha2020}
	{\sc L.~Sha, D.~Schonfeld, and J.~Wang}, {\em Graph {L}aplacian regularization
		with sparse coding for image restoration and representation}, IEEE
	Transactions on Circuits and Systems for Video Technology, 30 (2020),
	pp.~2000--2014.
	
	\bibitem{shefi}
	{\sc R.~Shefi and M.~Teboulle}, {\em Rate of convergence analysis of
		decomposition methods based on the proximal method of multipliers for convex
		minimization}, SIAM J. Optim., 24 (2014), pp.~269--297.
	
	\bibitem{Showalter}
	{\sc R.~E. Showalter}, {\em Monotone Operators in {B}anach Space and Nonlinear
		Partial Differential Equations}, vol.~49 of Mathematical Surveys and
	Monographs, American Mathematical Society, Providence, RI, 1997.
	
	\bibitem{svaiter}
	{\sc B.~F. Svaiter}, {\em On weak convergence of the {D}ouglas-{R}achford
		method}, SIAM J. Control Optim., 49 (2011), pp.~280--287.
	
	\bibitem{patrinos2020}
	{\sc A.~Themelis and P.~Patrinos}, {\em Douglas-{R}achford splitting and {ADMM}
		for nonconvex optimization: tight convergence results}, SIAM J. Optim., 30
	(2020), pp.~149--181.
	
	\bibitem{vu}
	{\sc B.~C. V\~{u}}, {\em A splitting algorithm for dual monotone inclusions
		involving cocoercive operators}, Adv. Comput. Math., 38 (2013), pp.~667--681.
	
	\bibitem{yin2016}
	{\sc M.~Yan and W.~Yin}, {\em Self equivalence of the alternating direction
		method of multipliers}, in Splitting Methods in Communication, Imaging,
	Science, and Engineering, Sci. Comput., Springer, Cham, 2016, pp.~165--194.
	
	\bibitem{Yang21}
	{\sc Y.~Yang, Y.~Tang, M.~Wen, and T.~Zeng}, {\em Preconditioned
		{D}ouglas-{R}achford type primal-dual method for solving composite monotone
		inclusion problems with applications}, Inverse Problems \& Imaging, 15
	(2021), pp.~787--825.
	
	\bibitem{Osher11}
	{\sc X.~Zhang, M.~Burger, and S.~Osher}, {\em A unified primal-dual algorithm
		framework based on {B}regman iteration}, J. Sci. Comput., 46 (2011),
	pp.~20--46.
	
\end{thebibliography}
\end{document}